\documentclass[11pt]{article}
\usepackage{definitions}
    
\begin{document}

\title{A Practical and Optimal First-Order Method for Large-Scale Convex Quadratic Programming}
\author{Haihao Lu\thanks{MIT, Sloan School of Management (haihao@mit.edu).} \and Jinwen Yang\thanks{University of Chicago, Department of Statistics (jinweny@uchicago.edu).}}

\date{}

\maketitle
 
\begin{abstract}
    Convex quadratic programming (QP) is an important class of optimization problem with wide applications in practice. The classic QP solvers are based on either simplex or barrier method, both of which suffer from the scalability issue because their computational bottleneck is solving linear equations. In this paper, we design and analyze a first-order method for QP, called restarted accelerated primal-dual hybrid gradient (rAPDHG), whose computational bottleneck is matrix-vector multiplication. We show that rAPDHG has a linear convergence rate to an optimal solution when solving QP, and the obtained linear rate is optimal among a wide class of primal-dual methods. Furthermore, we connect the linear rate with a sharpness constant of the KKT system of QP, which is a standard quantity to measure the hardness of a continuous optimization problem. Numerical experiments demonstrate that both restarts and acceleration can significantly improve the performance of the algorithm. Lastly, we present PDQP.jl, {an open-source solver based on rAPDHG that can be run on both GPU and CPU. With a numerical comparison with SCS and OSQP on standard QP benchmark sets and large-scale synthetic QP instances, we demonstrate the effectiveness of rAPDHG for solving QP.}
\end{abstract}

\section{Introduction}
Convex quadratic programming (QP) is a fundamental optimization problem class with wide applications in practice. For example, in finance, the Markowitz model serves as the cornerstone of contemporary portfolio optimization and is frequently solved as a QP in practical applications~\cite{markowitz1950theories,markowitz1967portfolio}. In statistical learning, Lasso~\cite{tibshirani1996regression}, trend filtering~\cite{kim2009ell_1}, support vector machine~\cite{cortes1995support}, among other models, are often formulated as a QP. In control, QP is an important element in applying model predictive control~\cite{bartlett2002quadratic}, just to name a few.

Classic QP solvers often use methods originally designed for solving linear programming (LP). For example, a variant of simplex method can be used on QP by taking advantage of the polytope nature of QP's feasible region~\cite{dantzig1998linear,ferreau2014qpoases}. In addition, an alternative method for QP solving involves interior-point methods (IPMs), which utilize barrier functions to penalize the linear constraints~\cite{vanderbei1999loqo,friedlander2012primal,andersen2003implementing}. Both the simplex and interior-point methods serve as foundational approaches in traditional QP solvers such as Gurobi~\cite{optimization2023gurobi} and MOSEK~\cite{aps2023mosek}, and they are quite reliable in providing solutions with high accuracy. More recently, there is a line of research on open-source solvers based on Alternating Direction Method of Multipliers (ADMM) for QP, such as SCS~\cite{o2016conic,o2021operator} and OSQP~\cite{stellato2020osqp}, which can also efficiently solve QP.

However, scaling up QP further using simplex methods or interior-point methods is highly challenging. The computational bottleneck of both methods is matrix factorization for solving linear systems, which has two significant drawbacks when solving large instances:  (1) Factorization is a sequential algorithm in nature, and it is highly challenging, if not impossible, to take advantage of parallel computing massively; (2) While the data in the objective and in the constraint may be very sparse, the factorization is usually much denser, and storing the factorization demands a considerable amount of memory. As a result, all of the commercial QP solvers are implemented on a single machine and are CPU-based, and cannot take advantage of modern computing architectures, such as GPUs and distributed computing. Furthermore, commercial solvers may often raise an ``out-of-memory'' error even though the problem instance can fit into memory when solving large instances. These concerns are also shared with ADMM-based solvers, which require the storage of one matrix factorization that can be repetitively reused. {ADMM-based solvers can also be matrix-free by using the conjugate gradient method for solving the linear equation to avoid matrix factorization, which is called the indirect ADMM method. Indirect ADMM methods are sometimes used to solve larger instances when the factorization is too expensive to be computed and/or stored, and they can be implemented on GPUs. Enhancements such as warm-start and preconditioning can be applied to the inner CG solver. However, they usually have inferior numerical performance to its factorization-based counterpart (see, for example, a comparison between SCS direct solve and indirect solve on LP~\cite{applegate2021faster}, and the experiments of QP in Section \ref{sec:PDQP}).}

To overcome such issues, there is a recent trend of research on using matrix-free first-order methods (FOMs) for scaling up LP~\cite{applegate2021practical,lin2021admm,deng2022new,basu2020eclipse}. In contrast to simplex and interior-point methods, the computational bottleneck of matrix-free FOMs is matrix-vector multiplication. Thanks to the recent development of deep learning, matrix-vector multiplication can scale very well on modern computing architectures such as GPUs and distributed computing. Furthermore, in terms of memory, FOMs only require storing the problem instance without storing any potentially denser factorizations, and thus, they can solve larger instances even just on a shared memory CPU machine. An example of this is the open-sourced LP solver PDLP~\cite{applegate2021practical,applegate2023faster,applegate2021infeasibility}. A distributed implementation of PDLP on 2000 machines solved an LP instance with 92 billion non-zeros in the constraint matrix, roughly 1000 times larger than commercial solvers' capability~\cite{blog}. A GPU implementation of PDLP (cuPDLP) has comparable performance as commercial solvers on medium-scale benchmark instances and has superior performance on larger instances~\cite{lu2023cupdlp,lu2023cupdlpc}.

{This paper aims to extend the success of PDLP to solving QP.
Fueled by the limitations encountered by simplex and IPMs and the recent theoretical development of FOMs, we pursue FOMs with the following four features that can be used as a base algorithm for a practical first-order solver:}
\begin{itemize}
    \item \textbf{Matrix-free}: The algorithm does not solve any linear equation, and the computational bottleneck is, at most, matrix-vector multiplication. 
    \item \textbf{Linear convergence}: The algorithm should converge to an optimal solution with at least a linear rate. This is because the users of solvers usually expect a high-accuracy solution, and a sublinear convergence rate is not enough to identify a high-accuracy solution.
    \item \textbf{Optimality}: We expect that the algorithm not only has linear convergence to an optimal solution, but the convergence rate also matches with an information-theoretical lower bound.
    \item \textbf{Strong numerical performance}: The algorithm should have reliable numerical performance, which can showcase evidence that a refined C or C++ implementation (with potential additional heuristics) may have comparable performance with state-of-the-art solvers on standard benchmark sets and superior performance on larger instances.
\end{itemize}

An example is the restarted primal-dual hybrid gradient (PDHG) algorithm for LP, which is the base algorithm in PDLP. The algorithm is matrix-free, achieves the optimal linear convergence rate for LP~\cite{applegate2023faster}, and has a strong numerical performance~\cite{applegate2021practical}. While restarted PDHG can be directly applied to QP, it is easy to see that it does not give the optimal linear convergence rate --  restarted gradient descent (a special case of restarted PDHG for unconstrained problems) is not an optimal algorithm for quadratic minimization (a special case of QP without constraints). Thus, restarted PDHG is not a desired FOM for QP. 

Of course, there can be multiple efficient FOMs for a class of optimization problems. We here aim to identify one of them. More formally, we consider the convex QP problem with the form
\begin{align}\label{eq:qp}
\begin{split}
    & \ \min_{x\in \mathbb R^n}\; \frac 12 x^TQx+c^Tx \\
    & \ \mathrm{\;s.t.\;}\; Ax\leq b \ ,
\end{split}
\end{align}
where objective vector $c\in \mathbb R^n$, right-hand-side $b\in \mathbb R^m$, constraint matrix $A\in \mathbb R^{m\times n}$ and objective matrix $Q\in\mathbb R^{n\times n}$ satisfying $Q\succeq 0$ due to convexity\footnote{Most of our theoretical and practical results in this paper can be extended to other forms of QP, and we utilize this form for the simplicity of theoretical development.}. Notice that it is genuinely hard to project onto the constraint set of \eqref{eq:qp}, we instead study the primal-dual formulation of the problem:
\begin{equation}\label{eq:minimax}
    \min_x\max_{y\geq 0}\; \mathcal L(x,y):= \frac 12 x^TQx+c^Tx+y^TAx-b^Ty \ .
\end{equation}
By duality theory in convex optimization, a saddle point to \eqref{eq:minimax} can recover an optimal primal-dual pair to the QP problem \eqref{eq:qp}; thus, we need to find a saddle point to \eqref{eq:minimax}.

How do we design an optimal FOM for QP? Intuitively, the matrix $Q$ is usually not full rank, and there are two orthogonal subspaces in the primal space: a linear subspace $\text{ker}(Q)$, and a quadratic subspace $\text{range}(Q)$. Along the linear subspace, the problem is essentially an LP, and restarted PDHG achieves the optimal rate. Along the quadratic subspace, the problem is quadratic; an optimal algorithm should utilize momentum/acceleration, for example, accelerated PDHG~\cite{chen2014optimal}. In this paper, we combine restarting and acceleration together, and propose restarted accelerated PDHG (rAPDHG, see Algorithm \ref{alg:rapd} for details) for solving QP, and show that it achieves an optimal linear convergence rate. There are a few technical challenges in the theoretical analysis:
\begin{itemize}
    \item Notice that QP is not strongly convex; thus, the accelerated PDHG algorithm~\cite{chen2014optimal} can only provide a sub-linear convergence rate. We show that restarting can help turn such sub-linear convergence rate into a linear convergence rate by utilizing the structure of QP.
    \item The linear convergence analysis of restarted PDHG for LP relies on the normalized duality gap being sharp for LP~\cite{applegate2023faster}. Unfortunately, the normalized duality gap, as defined in \cite{applegate2023faster}, is no longer sharp for QP due to the quadratic nature of the primal-dual gap. We here utilize a quadratic growth condition of a newly defined smoothed duality gap from \cite{fercoq2021quadratic} to overcome this issue.
    \item A careful analysis is proposed to ensure that the optimal convergence rates on the linear subspace and the quadratic subspace do not interfere with each other.
\end{itemize}

To demonstrate the numerical effectiveness of rAPDHG, we compare it with vanilla PDHG, restarted PDHG (without acceleration), and accelerated PDHG (without restarting), which showcase that both acceleration and restarting can significantly improve the performance, particularly for obtaining high-accuracy solutions. Furthermore, we develop a solver PDQP.jl based on rAPDHG and practical heuristics with both CPU and GPU implementations, and we demonstrate its superior performance for solving large-scale QP with an extensive numerical study.

The contributions of the paper can be summarized as follows:
\begin{itemize}
    \item We propose rAPDHG, a restarted accelerated primal-dual algorithm for solving convex QP.
    \item We show that the proposed algorithm enjoys linear convergence for solving convex QP, and we connect the convergence rate with the Hoffman constant of its KKT system, a quantity that is often used in characterizing the hardness of optimization problems.
    \item We present worst-case instances showing that the obtained linear convergence rate matches the information-theoretical lower bound for a wide class of first-order primal-dual algorithms. Thus, it is an optimal algorithm.
    \item We develop PDQP.jl, a QP solver in Julia based on rAPDHG and practical heuristics. We present an extensive numerical study of the proposed algorithm on a standard QP benchmark set, showcasing its strong performance compared with its FOM counterparts.
\end{itemize}

\subsection{Related literature}
{\bf Quadratic programming.} Quadratic programming, as a natural extension of linear programming, is a fundamental tool in engineering and operation research~\cite{bartlett2002quadratic,van2018large,rockafellar1987linear,markowitz1950theories,markowitz1967portfolio}. Two classic algorithms to solve QP problems are simplex methods~\cite{dantzig1998linear,ferreau2014qpoases} and interior-point methods~\cite{vanderbei1999loqo,friedlander2012primal,andersen2003implementing}. Based on these methods, commercial solvers such as Gurobi and MOSEK can provide reliable solutions with high accuracy. The success of simplex and IPMs depends on efficiently solving the associated linear systems, which makes them challenging to further scale up on modern computational architecture such as GPUs and distributed machines.

{\bf FOM-based solvers.} Recently, the research surges on first-order methods for solving large-scale optimization problems due to their low computational cost and ability of parallelization. Several solvers on LP have been developed and most of them can be extended to solve QP.

\begin{itemize}
    \item PDHG-based solvers: \href{https://github.com/google/or-tools/tree/stable/ortools/pdlp}{PDLP}~\cite{applegate2021practical}. PDLP is a general-purpose large-scale LP solver. The base algorithm of PDLP is a restarted PDHG algorithm~\cite{applegate2023faster}, and it leverages many practical algorithmic enhancements, such as presolving, preconditioning, adaptive restart, adaptive step size, on top of the base algorithm. One can show in theory that the base algorithm (restarted PDHG) is an optimal primal-dual algorithm for LP~\cite{applegate2023faster}. It currently has three implementations: a prototype implemented in Julia (\href{https://github.com/google-research/FirstOrderLp.jl}{FirstOrderLp.jl}), a production-level C++ implementation (open-sourced through \href{https://developers.google.com/optimization}{Google OR-Tools}), and an internal distributed version at Google.  Recently it is demonstrated that cuPDLP (\href{https://github.com/jinwen-yang/cuPDLP.jl}{cuPDLP.jl}/\href{https://github.com/COPT-Public/cuPDLP-C}{cuPDLP-C}), the GPU implementations of PDLP, has comparable performance as commercial solvers on medium-scale benchmark instances and has superior performance on larger instances~\cite{lu2023cupdlp,lu2023cupdlpc}. PDLP is completely matrix-free and is readily extended to tackle QP and other cone programs. Indeed, the C++ PDLP implementation also supports the solving of diagonal QP.
    \item Matrix-free IPM solvers: \href{https://github.com/leavesgrp/ABIP}{ABIP}~\cite{lin2021admm, deng2022new}. The core algorithm of ABIP is solving the homogeneous self-dual embedded cone programs via an interior-point method, and as a special case, it can solve LP and QP. ABIP utilizes multiple ADMM iterations instead of one Newton step to approximately minimize the log-barrier penalty function. A recent enhanced version of ABIP (named ABIP+~\cite{deng2022new}) includes many new enhancements, such as preconditioning, restart, hybrid parameter tuning, on top of ABIP. It was shown that ABIP+ (developed in C) has a comparable numerical performance on the LP benchmark sets to the Julia implementation of PDLP.
    \item ADMM-based solvers: \href{https://github.com/cvxgrp/scs}{SCS}~\cite{o2016conic,o2021operator} and \href{https://github.com/osqp/osqp}{OSQP}~\cite{stellato2020osqp}. SCS is designed to solve large-scale convex cone programs. It tackles the homogeneous self-dual embedding of general conic programming using ADMM. QP is solved as a special case of cone programming in the original version of SCS, and a recent update \cite{o2021operator} solves QP as linear complementarity problem by Douglas-Rachford splitting. OSQP is another popular general-purpose QP solver. In contrast to SCS, OSQP directly solves the original QP problem by ADMM without resorting to any embedding. The computational bottleneck of ADMM-based methods is solving a linear system with similar forms every iteration. There are two approaches for this: (i) Store the factorization of a matrix in memory and re-use it every iteration with minor change; (ii) Since the linear system is usually reasonably conditioned, unlike those in IPMs, one can use conjugate gradient method to solve the linear system in every iteration. SCS and OSQP support both solving approaches. The first approach still requires doing a factorization and thus suffers from the same scaling difficulties as IPMs. The second approach usually requires at least 20 conjugate gradient steps (i.e., matrix-vector multiplications) for every iteration and is not competitive compared with the PDHG-based algorithms as shown in LP~\cite{applegate2021practical}.
    \item {ALM-based solvers: \href{https://github.com/kul-optec/QPALM}{QPALM}~\cite{hermans2022qpalm}. QPALM is a solver for (possibly nonconvex) QP based on proximal augmented Lagrangian method. Though QPALM requires solving a sequence of linear systems, a fast linear algebra routines for the factorization is customized for QPALM to achieve strong numerical performances. Furthermore, R-linear convergence is also obtained theoretically.}
\end{itemize}

{\bf Restart scheme.} Restart is a technique to enhance the theoretical and practical performance without modification of base algorithms~\cite{pokutta2020restarting}. There have been extensive works on this scheme in smooth convex minimization~\cite{o2015adaptive, roulet2020sharpness}, nonsmooth convex minimization~\cite{freund2018new, yang2018rsg}, stochastic convex minimization~\cite{johnson2013accelerating, lin2015universal, tang2018rest}, and minimax optimization~\cite{applegate2023faster, lu2021s, zhao2022accelerated}. In the context of LP, \cite{applegate2023faster} proposes a fixed-frequency and an adaptive restart scheme for a large class of base primal-dual algorithms including PDHG and ADMM, and shows the restarted variant of base algorithms exhibits optimal linear convergence rate on LP. A restarted version of stochastic primal-dual methods is also discussed in \cite{lu2021s} for solving LP.

{\bf Momentum technique.} Momentum, or inertia, dates back to \cite{polyak1964some} and it is well-known to provide faster convergence rate. In particular, the original Polyak's method~\cite{polyak1964some} was shown to enjoy fast linear convergence rate on convex quadratic problems. Later, \cite{nesterov1983method} proposed a variant of momentum methods (which is now known as Nesterov's accelerated methods) and showed the optimality of this method for solving smooth convex problems~\cite{nesterov2003introductory,nemirovskii1985optimal,nemirovskij1983problem}. Since then, momentum-based algorithms have been extensively studied in optimization and machine learning~\cite{nesterov1983method,beck2009fast, kingma2014adam,sutskever2013importance,liu2020improved}, and many attempts explain the intuition behind momentum-based methods~\cite{su2014differential, bubeck2015geometric,allen2017linear}. Furthermore, \cite{chen2014optimal,goldstein2014fast,ouyang2015accelerated,xu2017accelerated} demonstrate the acceleration effect of momentum in the primal-dual settings. In particular, \cite{chen2014optimal} proposes an accelerated variant of PDHG while \cite{goldstein2014fast,ouyang2015accelerated} propose ADMM with momentum.

\subsection{Notations}
{Denote $z=(x,y)$ the primal-dual pair and let $\mathcal Z^*$ be the optimal solution set of \eqref{eq:minimax}.} Denote the ball centered at $z$ with radius $R$ as $B_R(z)=\{\hat z\;|\; \|\hat z-z\|\leq R\}$. Denote $\iota_{\mathcal C}$ the indicator function of set $\mathcal C$. For any $v\in \mathbb R^n$, let $[v]_+=[(v_1,...,v_n)]_+=([v_1]_+,...,[v_n]_+)$, where $[v_i]_+=\max\{v_i,0\}$, be the coordinate-wise positive part of a vector.

\section{Restarted Accelerated PDHG for Convex Quadratic Programming}\label{sec:RAPDHG}

In this section, we propose rAPDHG (Algorithm \ref{alg:rapd}) for solving convex quadratic programming \eqref{eq:minimax}, and present the linear convergence of the algorithm.

\begin{algorithm}[ht!]
        \SetAlgoLined
        {\bf Input:} Initial point $(x^{0,0},y^{0,0})$, step-size $\{(\beta_k,\theta_k,\eta_k,\tau_k)\}$, restart frequency $K$\;
            \Repeat{\textit{$(x^{n,0}, y^{n,0})$ converges}}{
            {\bf initialize the inner loop.} $(\bar x^{n,0},\bar y^{n,0}) \gets (x^{n,0}, y^{n,0})$\;
            \For{$k=0,...,K-1$}{
                $x_{md}^{n,k}=(1-\beta_k^{-1})\bar x^{n,k}+\beta_k^{-1}x^{n,k}$\;
                \vspace{0.075cm}
                $y^{n,k+1}=\mathrm{proj}_{\mathbb R^m_+}\left\{y^{n,k}+\tau_k(A(\theta_{k}(x^{n,k}-x^{n,k-1})+x^{n,k})-b)\right\}$\;
                \vspace{0.075cm}
                $x^{n,k+1}=x^{n,k}-\eta_k\pran{Qx_{md}^{n,k}+c+A^Ty^{n,k+1}}$\;
                \vspace{0.075cm}
                $\bar x^{n,k+1}=(1-\beta_k^{-1})\bar x^{n,k}+\beta_k^{-1}x^{n,k+1}$\;
                \vspace{0.075cm}
                $\bar y^{n,k+1}=(1-\beta_k^{-1})\bar y^{n,k}+\beta_k^{-1}y^{n,k+1}$\;
            }
        {\bf restart the outer loop.} $(x^{n+1,0}, y^{n+1,0})\gets (\bar x^{n,K},\bar y^{n,K})$, $n\gets n+1$\;
            }
        {\bf Output:} $(x^{n,0}, y^{n,0})$.
        \caption{Restarted accelerated PDHG (rAPDHG) on \eqref{eq:minimax}}
        \label{alg:rapd}
\end{algorithm}

Algorithm \ref{alg:rapd} presents our two-loop restarted accelerated PDHG. We start from an initial solution $(x^{0,0}, y^{0,0})$, a set of step-size parameters, and a restart frequency $K$. We denote $(x^{n,k},y^{n,k})$ the iterate at the $k$-th inner loop of the $n$-th outer loop and denote the running average of the $n$-th outer loop as $(\bar x^{n,k},\bar y^{n,k})$. The inner loop essentially runs the accelerated PDHG algorithm developed in \cite{chen2014optimal}, where at the high level, the primal variable $x^{n,k}$ is updated by an accelerated gradient descent step, and the dual variable $y^{n,k}$ is updated by a standard PDHG dual step. After running $K$ iterations, we terminate the inner loop and restart the next outer loop from the the running average of the current inner loop $(\bar x^{n,K},\bar y^{n,K})$.

In the rest of this section, we present computational guarantees of Algorithm \ref{alg:rapd}. We start by introducing two progress metrics that are essential in the follow-up analysis:
\begin{mydef}[Duality gap]
{Denote $z=(x,y)$ the primal-dual pair.} For any $z,\hat z\in \mathcal Z$, denote
    \begin{equation*}
        Q(z,\hat{z}) := \mathcal L(x,\hat y)-\mathcal L(\hat x,y) \ .
    \end{equation*}
We call $\max_{\hat z\in\mathcal Z} Q(z,\hat z)$ the duality gap at $z$.
\end{mydef}

The duality gap is a natural progress metric for studying primal-dual algorithms and has been used in many related literature. However, the duality gap is not a good progress metric for QP, because it often is infinity due to the potentially unbounded constraint set. To overcome this issue, \cite{applegate2023faster} proposes a normalized duality gap for studying LP, which normalizes the duality gap by the distance between $z$ and $\hat{z}$. The normalized duality gap is a natural progress metric for studying primal-dual algorithms with unbounded feasible regions.  Unfortunately, the normalized duality gap is not sharp for QP due to the existence of the quadratic term, and thus the analysis in \cite{applegate2023faster} cannot be used for analyzing QP. Dual to the normalized duality gap, \cite{fercoq2021quadratic} proposes a smoothed duality gap that penalizes the distance $z$ and $\hat{z}$, which is defined below:

\begin{mydef}[Smoothed duality gap \cite{fercoq2021quadratic}]
For any $\xi\geq 0$ and $z,\dot z\in \mathcal Z$, we define the smoothed duality gap at $z$ centered at $\dot z$ as
    \begin{equation*}
        G_{\xi}(z;\dot{z}) := \max_{\hat z\in \mathcal Z} \left\{Q(z,\hat{z})-\frac{\xi}{2}\| \hat z-\dot{z}\|^2\right\}=\max_{\hat z\in \mathcal Z}\left\{\mathcal L(x,\hat y)-\mathcal L(\hat x,y)-\frac{\xi}{2}\| \hat z-\dot{z}\|^2\right\} \ .
    \end{equation*}
\end{mydef}

Similar to the normalized duality gap, the smoothed duality gap always retains a finite value. Furthermore, \cite{fercoq2021quadratic} introduces the quadratic growth of smoothed duality gap, which plays a major role in their analysis of the linear convergence of PDHG~\cite{fercoq2021quadratic}. We here utilize this concept to study Algorithm \ref{alg:rapd}:

\begin{mydef}[\cite{fercoq2021quadratic}]
Suppose $\xi>0$. We say that the smoothed duality gap satisfies quadratic growth on set $\mathbb S$ if it holds for all $z^*\in \mathcal Z^*$ and any $z\in \mathbb S$ that 
    \begin{equation*}
        G_{\xi}(z;z^*) \geq \alpha_{\xi} \mathrm{dist}^2(z,\mathcal Z^*) \ ,
    \end{equation*}
    where $\alpha_\xi>0$.
\end{mydef}

It turns out that the smoothed duality gap of QP \eqref{eq:qp} satisfies quadratic growth on bounded sets, and we present a formal characterization of the parameter $\alpha_{\xi}$ for QP in Section \ref{sec:growth-qp}. 

The next theorem presents our major theoretical result of Algorithm \ref{alg:rapd} for QP. 

\begin{thm}[Linear Convergence]\label{thm:main}
Consider the sequence $\{z^{n,0}\}_{n=0}^\infty$ generated by rAPDHG (Algorithm \ref{alg:rapd}) with fixed restart frequency for solving \eqref{eq:minimax}. Suppose for any $\xi>0$, \eqref{eq:minimax} satisfies quadratic growth with parameter $\alpha_\xi$ on a ball $B_R(z^{0,0})$ with center $z^{0,0}$ and radius $R=\frac{3}{1-1/e}\mathrm{dist}(z^{0,0},\mathcal Z^*)$. For each inner iteration $0\leq k\leq K-1$, let 
\begin{equation*}
        \beta_k = \frac{k+2}{2},\; \theta_k=\frac{k}{k+1},\; \eta_k = \frac{k+1}{2\pran{\|Q\|+K\|A\|}},\; \tau_k = \frac{k+1}{2K\|A\|} \ .
\end{equation*}
Suppose the restart frequency $K$ satisfies
\begin{equation*}
    K\geq \max\left\{ \sqrt{\frac{32e^2\|Q\|}{\alpha_\xi}}, \frac{32e^2\|A\|}{\alpha_{\xi}}, \sqrt{\frac{64\|Q\|}{\xi}}, \frac{64\|A\|}{\xi},  \frac{\|Q\|}{\|A\|}  \right\} \ .
\end{equation*}
Then it holds for any outer iteration $n$ that
\begin{enumerate}
    \item[(i)] $\{z^{n,0}\}$ is bounded and stays in $B_R(z^{0,0})$:
    \begin{equation*}
        \|z^{n,0}-z^{0,0}\|\leq \frac{3}{1-1/e}\mathrm{dist}(z^{0,0},\mathcal Z^*) \ ,
    \end{equation*}
    \item[(ii)] $\{z^{n,0}\}$ converges linearly to an optimal solution:
    \begin{equation*}
        \mathrm{dist}(z^{n,0},\mathcal Z^*) \leq e^{-n}\mathrm{dist}(z^{0,0},\mathcal Z^*) \ . 
    \end{equation*}
\end{enumerate}  
\end{thm}

    Theorem \ref{thm:main} shows that we achieve an $\epsilon$-close solution to \eqref{eq:minimax} in the sense of distance to optimality after $$\mathcal O\pran{ \max\left\{ \sqrt{\frac{\|Q\|}{\alpha_{\xi}}}, \frac{\|A\|}{\alpha_{\xi}} , \sqrt{\frac{\|Q\|}{\xi}}, \frac{\|A\|}{\xi},  \frac{\|Q\|}{\|A\|}  \right\} \log\frac{\mathrm{dist}(z^{0,0},\mathcal Z^*)}{\epsilon} }$$ total iterations. Notice that the theoretical result holds for any parameter $\xi$, which only affects the restarting frequency in the algorithm. One can tune the value of $\xi$ to obtain a better complexity. Section \ref{sec:lower_bound} presents lower bound instances showing that such complexity is the best one can hope within a wide class of first-order primal-dual algorithms when the optimal $\xi$ is chosen.

The rest of this section presents the proof of Theorem \ref{thm:main}. We present three technical lemmas that will be used in the proof. These three lemmas only involve the inner iterations. For simplicity, we drop the counter $n$ of the outer loop in the presentation and proof of Lemma \ref{lem:apd}-\ref{lem:sublinear}.

First, we present a decay lemma for accelerated PDHG established in \cite{chen2014optimal}, where part (i) builds a telescoping bound of $Q(\bar z_{k+1},\hat z)$ while part (ii) basically shows the iterates of accelerated PDHG are bounded.
\begin{lem}\cite[Lemma 4.2 and 4.3]{chen2014optimal}\label{lem:apd}
Let $\bar z_{k+1}=(\bar x_{k+1},\bar y_{k+1})$ be the iterates generated by accelerated PDHG, {namely
\begin{equation*}
    \begin{aligned}
        & x^{md}_{k}=(1-\beta_k^{-1})\bar x_{k}+\beta_k^{-1}x_{k}\\
        & y_{k+1}=\mathrm{proj}_{\mathbb R^m_+}\left\{y_{k}+\tau_k(A(\theta_{k}(x_{k}-x_{k-1})+x_{k})-b)\right\}\\
        & x_{k+1}=x_{k}-\eta_k\pran{Qx^{md}_{k}+c+A^Ty_{k+1}}\\
        & \bar x_{k+1}=(1-\beta_k^{-1})\bar x_{k}+\beta_k^{-1}x_{k+1}\\
        & \bar y_{k+1}=(1-\beta_k^{-1})\bar y_{k}+\beta_k^{-1}y_{k+1} \ .
    \end{aligned}
\end{equation*}
}

Assume that the parameters $\beta_k$, $\theta_k$, $\eta_k$, $\tau_k$ are chosen such that
\begin{equation}\label{eq:param}
    \beta_0=1,\;\beta_{k+1}-1=\beta_k\theta_{k+1},\;0<\theta_{k+1}\leq \min\left\{ \frac{\eta_{k}}{\eta_{k+1}},\frac{\tau_{k}}{\tau_{k+1}} \right\},\; \frac{1}{\eta_k}-\frac{\|Q\|}{\beta_k}-4{\|A\|^2\tau_k} \geq 0 \ .
\end{equation}
Denote $\gamma_k=\theta_k^{-1}\gamma_{k-1}$, $\gamma_0=1$. Then 

(i) It holds for all $k\geq 0$ and any $\hat z=(\hat x,\hat y)\in\mathcal Z$ that
    \begin{align*}
    \begin{split}
        \beta_k\gamma_kQ(\bar z_{k+1},\hat z) & \ \leq \frac 12\sum_{i=0}^k \frac{\gamma_i}{\eta_i}\pran{\|\hat x -x_i\|^2-\|\hat x -x_{i+1}\|^2}+\frac{\gamma_i}{\tau_i}\pran{\|\hat y -y_i\|^2-\|\hat y -y_{i+1}\|^2} \\ 
        & \quad\ -\gamma_k\langle A(x_{k+1}-x_k),\hat y-y_{k+1}\rangle -\gamma_k\pran{\frac{1}{2\eta_k}-\frac{\|Q\|}{2\beta_k}}\|x_{k+1}-x_k\|^2 \ .
    \end{split}
    \end{align*}

(ii) It holds for all $k\geq 0$ and any $z^*=(x^*,y^*)\in \mathcal Z^*$ that
\begin{equation*}
        \|x^*-x^{k}\|^2+\frac 34\frac{\eta_k}{\tau_k}\|y^*-y^k\|^2 \leq \|x^*-x^{0}\|^2+\frac{\eta_k}{\tau_k}\|y^*-y^0\|^2 \ .
\end{equation*}
\end{lem}
The following lemma derives the sublinear rate of the accelerated PDHG iterates for primal-dual gap with specific choices of parameters.
\begin{lem}\label{lem:stepsize}
    Let $\bar z_{k+1}=(\bar x_{k+1},\bar y_{k+1})$ be the iterates generated by accelerated PDHG. Suppose the total iteration count $K$ is fixed. For $0\leq k\leq K-1$, denote
    \begin{equation}\label{eq:param-choice}
        \beta_k = \frac{k+2}{2},\; \theta_k=\frac{k}{k+1},\; {\eta_k = \frac{k+1}{2\pran{\|Q\|+K\|A\|}},\; \tau_k = \frac{k+1}{2K\|A\|} \ .}
    \end{equation}
    Then
    \begin{enumerate}
        \item[(i)] Condition \eqref{eq:param} holds, and furthermore, $\theta_{k+1}=\frac{\eta_k}{\eta_{k+1}}=\frac{\tau_k}{\tau_{k+1}}$.
        \item[(ii)] It holds for any $\hat z\in\mathcal Z$ that
        \begin{equation*}
            Q(\bar z_{K},\hat z) \leq \pran{\frac{8\|Q\|}{K^2}+\frac{8\|A\|}{K}}\|z_0-\hat z\|^2 \ .
        \end{equation*}
        \item[(iii)] It holds for any $0\leq k\leq K-1$ with $K\geq \frac{\|Q\|}{\|A\|}$ that
        \begin{equation*}
            \frac{1}{4}\|z^k-z^{*}\|^2 \leq \|z^0-z^{*}\|^2 \ .
        \end{equation*}
    \end{enumerate}
\end{lem}

\begin{proof}
(i) It is straightforward to check the first two conditions in \eqref{eq:param} indeed hold for the choice of $\beta_k$ and $\theta_k$ in \eqref{eq:param-choice}. Note that the total iteration number $K$ is fixed and we have $\frac{\eta_k}{\eta_{k+1}}=\frac{\tau_k}{\tau_{k+1}}=\frac{k+1}{k+2}=\theta_{k+1}$, which in turn guarantees the third condition in \eqref{eq:param}. As for the final condition, with $0\leq k\leq K-1$,
\begin{equation*}
    \begin{aligned}
        \frac{1}{\eta_k}-\frac{\|Q\|}{\beta_k}-{4\|A\|^2\tau_k} = \frac{2\|Q\|}{(k+1)(k+2)}+2\|A\|\pran{\frac{K}{k+1}-\frac{k+1}{K}}\geq 0 \ .
    \end{aligned}
\end{equation*}

(ii)
By Lemma \ref{lem:apd}, we have
    \begin{align*}
        \begin{split}
            \beta_kQ(\bar z_{k+1},\hat z) & \ \leq \frac{1}{2\gamma_k}\left\{ \frac{\gamma_0}{\eta_0}\|\hat x-x_0\|^2+\sum_{i=0}^{k-1} \pran{\frac{\gamma_{i+1}}{\eta_{i+1}}-\frac{\gamma_i}{\eta_i}}\|\hat x-x_{i+1}\|^2-\frac{\gamma_k}{\eta_k} \|\hat x-x_{k+1}\|^2 \right\} \\
            & \quad\quad + \frac{1}{2\gamma_k}\left\{ \frac{\gamma_0}{\tau_0}\|\hat y-y_0\|^2+\sum_{i=0}^{k-1} \pran{\frac{\gamma_{i+1}}{\tau_{i+1}}-\frac{\gamma_i}{\tau_i}}\|\hat y-y_{i+1}\|^2-\frac{\gamma_k}{\tau_k} \|\hat y-y_{k+1}\|^2 \right\} \\
            & \quad\quad -\langle A(x_{k+1}-x_k),\hat y-y_{k+1}\rangle -\pran{\frac{1}{2\eta_k}-\frac{\|Q\|}{2\beta_k}}\|x_{k+1}-x_k\|^2 \ .
        \end{split}
    \end{align*}
    {Since $\frac{\gamma_{i+1}}{\gamma_i}=\frac{1}{\theta_{i+1}}=\frac{\eta_{i+1}}{\eta_i}=\frac{\tau_{i+1}}{\tau_i}$}, it holds that 
    \begin{align*}
        \begin{split}
            \beta_kQ(\bar z_{k+1},\hat z) & \ \leq \frac{1}{2\gamma_k}\frac{\gamma_0}{\eta_0}\|\hat x-x_0\|^2 + \frac{1}{2\gamma_k} \frac{\gamma_0}{\tau_0}\|\hat y-y_0\|^2-\frac{1}{2\tau_k} \|\hat y-y_{k+1}\|^2 \\
            & \quad\quad -\langle A(x_{k+1}-x_k),\hat y-y_{k+1}\rangle -\pran{\frac{1}{2\eta_k}-\frac{\|Q\|}{2\beta_k}}\|x_{k+1}-x_k\|^2\\
            & \ = \frac{1}{2\eta_k}\|\hat x-x_0\|^2 + \frac{1}{2\tau_k}\|\hat y-y_0\|^2-\frac{1}{2\tau_k} \|\hat y-y_{k+1}\|^2 \\
            & \quad\quad -\langle A(x_{k+1}-x_k),\hat y-y_{k+1}\rangle -\pran{\frac{1}{2\eta_k}-\frac{\|Q\|}{2\beta_k}}\|x_{k+1}-x_k\|^2\ , 
        \end{split}
    \end{align*}
    where the second inequality uses $\frac{\gamma_0}{\gamma_k}=\frac{\gamma_0}{\gamma_1}\frac{\gamma_1}{\gamma_2}...\frac{\gamma_{k-1}}{\gamma_k}=\frac{\eta_0}{\eta_k}=\frac{\tau_0}{\tau_k}$.
    
    Note that $-\langle A(x_{k+1}-x_k),\hat y-y_{k+1}\rangle \leq \frac{\tau_k\|A\|^2}{2}\|x_{k+1}-x_k\|^2+ \frac{1}{2\tau_k}\|\hat y-y_{k+1}\|^2$,
    \begin{align*}
        \begin{split}
            \beta_kQ(\bar z_{k+1},\hat z)             & \ \leq \frac{1}{2\eta_k}\|\hat x-x_0\|^2 + \frac{1}{2\tau_k}\|\hat y-y_0\|^2-\frac{1}{2\tau_k} \|\hat y-y_{k+1}\|^2 \\
            & \quad\quad +\frac{\tau_k\|A\|^2}{2}\|x_{k+1}-x_k\|^2+ \frac{1}{2\tau_k}\|\hat y-y_{k+1}\|^2 -\pran{\frac{1}{2\eta_k}-\frac{\|Q\|}{2\beta_k}}\|x_{k+1}-x_k\|^2\\
            & \ = \frac{1}{2\eta_k}\|\hat x-x_0\|^2 + \frac{1}{2\tau_k}\|\hat y-y_0\|^2 - \pran{\frac{1}{2\eta_k}-\frac{\|Q\|}{2\beta_k}-\frac{\tau_k\|A\|^2}{2}}\|x_{k+1}-x_k\|^2 \\ 
            & \ \leq \frac{1}{2\eta_k}\|\hat x-x_0\|^2 + \frac{1}{2\tau_k}\|\hat y-y_0\|^2 \ ,
        \end{split}
    \end{align*}
    where the last inequality follows from $\frac{1}{\eta_k}-\frac{\|Q\|}{\beta_k}-\tau_k\|A\|^2\geq 0$.
    Finally, set $k=K-1$ and we achieve
    \begin{align*}
        \begin{split}
            Q(\bar z_{K},\hat z) & \ \leq \frac{1}{2}\pran{\frac{1}{\beta_{K-1}\eta_{K-1}}\|\hat x-x_0\|^2+\frac{1}{\beta_{K-1}\tau_{K-1}}\|\hat y-y_0\|^2} \\
            & \ = 2\pran{\frac{\|Q\|+K\|A\|}{K(K+1)}+\frac{\|A\|}{K+1}}\|\hat z-z_0\|^2 \\
            & \ \leq 8\pran{\frac{\|Q\|}{K^2}+\frac{\|A\|}{K}}\|\hat z-z_0\|^2 \ .
        \end{split}
    \end{align*}    

    (iii) Plug the choice of stepsizes in \eqref{eq:param-choice} into part (ii) in Lemma \ref{lem:apd}. We have
    \begin{equation}\label{eq:norm-decay}
        \|x^{k}-x^*\|^2+\frac{3}{4}\frac{K\|A\|}{K\|A\|+\|Q\|}\|y^k-y^*\|^2 \leq \|x^{0}-x^*\|^2+\frac{K\|A\|}{K\|A\|+\|Q\|}\|y^0-y^*\|^2  \ .
    \end{equation}
    Note that $K\geq \frac{\|Q\|}{\|A\|}$ and thus $\frac{K\|A\|}{K\|A\|+\|Q\|}\geq \frac{1}{2}$. It holds that
    {\small
    \begin{equation*}
    \begin{aligned}
        \frac{1}{4}\|z^{k}-z^*\|^2&=\frac{1}{4}\|x^k-x^*\|^2+\frac 14 \|y^k-y^*\|^2\leq \|x^k-x^*\|^2 + \frac{3}{4}\frac{K\|A\|}{K\|A\|+\|Q\|}\|y^k-y^*\|^2\\
        & \leq \|x^{0}-x^*\|^2+\frac{K\|A\|}{K\|A\|+\|Q\|}\|y^0-y^*\|^2 \leq \|z^0-z^*\|^2 \ ,
    \end{aligned}
    \end{equation*}
    }
    where the second inequality follows from \eqref{eq:norm-decay}.
\end{proof}

The next lemma indicates that the smoothed duality gap exhibits the same sublinear convergence rate as the primal-dual gap.
\begin{lem}\label{lem:sublinear}
Let $\bar z_{k+1}=(\bar x_{k+1},\bar y_{k+1})$ be the iterates generated by accelerated PDHG. Let $\xi>0$. Suppose the total iteration count $K$ is fixed and satisfies
\begin{equation*}
    K\geq \max\left\{ \sqrt{\frac{64\|Q\|}{\xi}}, \frac{64\|A\|}{\xi}  \right\} \ .
\end{equation*}
Let $0\leq k\leq K-1$ and
\begin{equation*}
        \beta_k = \frac{k+2}{2},\; \theta_k=\frac{k}{k+1},\; {\eta_k = \frac{k+1}{2\pran{\|Q\|+K\|A\|}},\; \tau_k = \frac{k+1}{2K\|A\|} \ .}
\end{equation*}
Then it holds for any $\dot z\in\mathcal Z$ that
    \begin{equation*}
        G_{\xi}(\bar z_K;\dot{z})\leq 16\pran{\frac{\|Q\|}{K^2}+\frac{\|A\|}{K}}\|z_0-\dot{z}\|^2 \ .
    \end{equation*}
\end{lem}

\begin{proof}
By Lemma \ref{lem:stepsize}, we have
    \begin{align}\label{eq:eq-rhs}
        \begin{split}
            Q(\bar z_K,\hat z)-\xi \|\hat{z}-\dot{z}\|^2 & \ \leq 8\pran{\frac{\|Q\|}{K^2}+\frac{\|A\|}{K}}\|\hat z-z_0\|^2-\frac{\xi}{2} \|\hat{z}-\dot{z}\|^2\ .
        \end{split}
    \end{align}
    With the choice of $K\geq \max\left\{ \sqrt{\frac{64\|Q\|}{\xi}}, \frac{64\|A\|}{\xi}  \right\}$, we have $8\pran{\frac{\|Q\|}{K^2}+\frac{\|A\|}{K}}<\frac{\xi}{2}$ and thus the right-hand-side of \eqref{eq:eq-rhs} attain its maximum at $\hat z=\frac{\frac{\xi}{2}\dot{z}-8\pran{\frac{\|Q\|}{K^2}+\frac{\|A\|}{K}}z^0}{\frac{\xi}{2}-8\pran{\frac{\|Q\|}{K^2}+\frac{\|A\|}{K}}}$. Therefore,
    \begin{align*}
        \begin{split}
            Q(\bar z_K,\hat z)-\xi \|\hat{z}-\dot{z}\|^2 & \ \leq 8\pran{\frac{\|Q\|}{K^2}+\frac{\|A\|}{K}}\|\hat z-z_0\|^2-\frac{\xi}{2} \|\hat{z}-\dot{z}\|^2 \leq \frac{8\xi\pran{\frac{\|Q\|}{K^2}+\frac{\|A\|}{K}}}{\xi-16\pran{\frac{\|Q\|}{K^2}+\frac{\|A\|}{K}}}\|z_0-\dot{z}\|^2 \ .
        \end{split}
    \end{align*}
    
    Note that 
    \begin{equation*}
        \xi-16\pran{\frac{\|Q\|}{K^2}+\frac{\|A\|}{K}} \geq \frac{\xi}{2} \ ,
    \end{equation*}
    and therefore
    \begin{align*}
        \begin{split}
            G_{\xi}(\bar z_K;\dot{z})=\max_{\hat z\in \mathcal Z}\left\{Q(\bar z_K,\hat z)-\xi \|\hat{z}-\dot{z}\|^2\right\}\leq 16\pran{\frac{\|Q\|}{K^2}+\frac{\|A\|}{K}}\|z_0-\dot{z}\|^2 \ .
        \end{split}
    \end{align*}
\end{proof}

Now we are ready to prove Theorem \ref{thm:main}.
\begin{proof}[Proof of Theorem \ref{thm:main}]
We prove by induction. The results hold for $t=0$. Suppose it holds for $t\leq D$. Denote $z^*_t=\mathrm{argmin}_{z\in \mathcal Z^*}\|z^{t,0}-z\|^2$
\begin{align*}
    \begin{split}
        \|z^{D+1,0}-z^{0,0}\| & \leq \sum_{t=0}^D \|z^{t+1,0}-z^{t,0}\| \leq \sum_{t=0}^D \|z^{t+1,0}-z^*_t\|+\|z^*_t-z^{t,0}\|\leq \sum_{t=0}^D 3\|z^*_t-z^{t,0}\| \\ 
        & = \sum_{t=0}^D 3\mathrm{dist}(z^{t,0},\mathcal Z^*) \leq 3\sum_{t=0}^D e^{-t}\mathrm{dist}(z^{0,0},\mathcal Z^*) \leq \frac{3}{1-1/e}\mathrm{dist}(z^{0,0},\mathcal Z^*) \ ,
    \end{split}
\end{align*}
where the third inequality uses part (iii) in Lemma \ref{lem:stepsize} by noticing that $K\geq \tfrac{\|Q\|}{\|A\|}$. This implies that $z^{D+1,0}$ is in set $B_{R}(z^{0,0})$ and thus $\alpha_{\xi}$ is a valid constant of quadratic growth for iterates $z^{D+1,0}$. Let $z^*=\mathrm{argmin}_{z\in \mathcal Z^*}\|z^{D,0}-z\|^2$ and we have
    \begin{align*}
        \mathrm{dist}^2(z^{D+1,0},\mathcal Z^*) &  \leq \frac{1}{\alpha_{\xi}} G_{\xi}(z^{D+1,0};z^*)=\frac{1}{\alpha_{\xi}} G_{\xi}(z^{D,K};z^*) \\ 
        & \ \leq \frac{16}{\alpha_{\xi}}\pran{\frac{\|Q\|}{K^2}+\frac{\|A\|}{K}}\|z^{D,0}-z^*\|^2 \\ 
        & \ = \frac{16}{\alpha_{\xi}}\pran{\frac{\|Q\|}{K^2}+\frac{\|A\|}{K}}\mathrm{dist}^2(z^{D,0},\mathcal Z^*) \\ 
        & \ \leq \frac{1}{e^2} \mathrm{dist}^2(z^{D,0},\mathcal Z^*) \\ 
        & \ \leq e^{-2(D+1)}\mathrm{dist}^2(z^{0,0},\mathcal Z^*) \ ,
    \end{align*}
    where the first inequality exploits the definition of quadratic growth and the second one follows from Lemma \ref{lem:sublinear}. The second equality is due to the choice of $z^*=\mathrm{argmin}_{z\in \mathcal Z^*}\|z^{D,0}-z\|^2$ and the third inequality uses the choice of $K$. By induction we derive the boundedness and linear convergence of $\{z^{n,0}\}$.
\end{proof}

\section{Understanding $\alpha_\xi$ via the KKT residual of QP}\label{sec:growth-qp}
Section \ref{sec:RAPDHG} shows that the performance of rAPDHG relies on the quadratic growth constant $\alpha_\xi$. In this section, we build the connection between growth constant $\alpha_\xi$ and the KKT residual of QP, which is often used to characterize the hardness of an optimization problem (for example, a similar quantity is used to characterize the performance of different FOMs for solving LP~\cite{applegate2023faster}).

The KKT conditions provide an optimality certificate of a solution in continuous optimization. The KKT condition of QP problem \eqref{eq:qp} is
\begin{equation}\label{eq:KKT}
    \begin{aligned}
        & \text{(primal feasibility):} &&~ Ax\le b, \\
        & \text{(dual feasibility):} && ~Qx+c+A^Ty=0, \\
        & \text{(complementary slackness):} &&~ y_i(b-Ax)_i =0 ,\ \forall i =1,...,m.
    \end{aligned}
\end{equation}
In other words, if a primal-dual solution pair $z=(x,y)$ satisfies \eqref{eq:KKT}, it is an optimal primal-dual solution pair. Notice that the complementary slackness is quadratic, which sometimes can be problematic in deriving theoretical results. On the other hand, it is straight-forward to check that the complementary slack is indeed equivalent to the following condition that avoids the quadratic formulation:
\begin{equation}\label{eq:fixedpoint}
    y-\left[y-\eta(b-Ax)\right]_+=0 \ , \text{~for any ~} \eta>0\ .
\end{equation}
\eqref{eq:fixedpoint} essentially states that $y$ is a fixed point to one iteration of dual projected gradient descent step with step-size $\eta$. 

Next, we introduce a scaled version of KKT condition with \eqref{eq:fixedpoint} that plays a major role in our analysis in this section:

\begin{mydef}
    We call
    \begin{equation*}
        F_\xi(z) = F_\xi(x,y) = \begin{pmatrix} [Ax-b]_+ \\ \pran{I+\frac{1}{\xi}Q}^{-1/2} (Qx+c+A^Ty) \\ y-\left[y-\frac{1}{\xi}(b-Ax)\right]_+   \end{pmatrix} \ .
    \end{equation*}
    the scaled KKT residual of \eqref{eq:qp}.
\end{mydef}

The scaled KKT residual $F_\xi(z)$ provides an optimality proxy of the solution $z$, namely, $z=(x,y)$ is an optimal primal-dual pair of \eqref{eq:qp} if and only if $F_\xi(z)=0$, and generally speaking, the small the value of $F_\xi(z)$, the closer $z$ to an optimal solution. Notice that every entry in $F_\xi(z)$ is a piecewise linear function in $z$, thus the $\ell_2$ norm of $F_\xi(z)$ is a sharp function, name, there exists $\gamma_\xi>0$ such that it holds for any $z$ with $\|z\|\leq R$ that
\begin{align}\label{eq:gamma}
    \gamma_{\xi}\mathrm{dist}(z,\mathcal Z^*) & \leq \|F_{\xi}(z)\|= \sqrt{\|[Ax-b]_+\|^2+\|Qx+c+A^Ty\|_{\pran{I+\frac{1}{\xi}Q}^{-1}}^2+\left\|y-\left[y-\frac{1}{\xi}(b-Ax)\right]_+\right\|^2} \ . 
\end{align}

In the definition of $F_\xi(z)$, we rescale the dual infeasibility with matrix $(I+\frac{1}{\xi}Q)^{-1}$, and utilizes $\eta=\frac{1}{\xi}$ in \eqref{eq:fixedpoint}. Both choices are natural when building up the connection between $\gamma_\xi$ and the quadratic growth constant $\alpha_\xi$, as stated in the following theorem:

\begin{thm}\label{thm:qp-qg}
    For any $\xi>0$, the smoothed duality gap of \eqref{eq:minimax} satisfies quadratic growth on $\mathcal Z\cap B_R(0)$ with constant $\frac{\gamma_\xi^2}{4\xi+2\|Q\|+3/\xi}>0$,
    where $\gamma_\xi$ is defined in \eqref{eq:gamma}. Namely, it holds for any $z\in\mathcal Z$ with $\|z\|\leq R$ that
    \begin{equation*}
        \frac{\gamma_\xi^2}{4\xi+2\|Q\|+3/\xi}\mathrm{dist}^2(z,\mathcal Z^*)\leq G_{\xi}(z;z^*) \ .
    \end{equation*}
\end{thm}

Theorem \ref{thm:qp-qg} connects the quadratic growth constant $\alpha_\xi$ with the sharpness of the scaled residual $F_{\xi}(z)$. Together with Theorem \ref{thm:main}, it shows that the linear convergence of Algorithm \ref{alg:rapd} is driven by the sharpness of KKT residual, which is consistent with the recent findings for LP~\cite{applegate2023faster}. 

The proof of Theorem \ref{thm:qp-qg} will be presented in the rest of the section. We first obtain the closed form of the smoothed duality gap of \eqref{eq:minimax} through direct calculation.
\begin{lem}
Let $z^*$ be an optimal solution to \eqref{eq:minimax}. Denote $G_\xi(z;z^*)$ the smoothed duality gap of \eqref{eq:minimax} where $z\in \mathcal Z$. Then it holds for any $\xi>0$ that
{
\begin{align}\label{eq:gap-explicit}
\begin{split}
    G_{\xi}(z;z^*)& \ =\frac{1}{2}x^TQx+c^Tx+b^Ty+\iota_{\mathbb R^m_+}(y) \\
    & \ \quad+\frac{\xi}{2}\left\| \left[ y^*+\frac{1}{\xi}(Ax-b)\right]_+\right\|^2-\frac{\xi}{2}\|y^*\|^2\\
    & \ \quad+\frac{\xi}{2}\left\| x^*-\frac{1}{\xi}(c+A^Ty)\right\|_{\pran{I+\frac{1}{\xi}Q}^{-1}}^2-\frac{\xi}{2}\|x^*\|^2 \ .
\end{split}    
\end{align}
}
\end{lem}
\begin{proof}
By definition, we can rewrite the smoothed duality gap as follows
{
    \begin{align}\label{eq:gap-rewrite}
    \begin{split}
        G_{\xi}(z;z^*)& \ =\max_{\hat z}\ \left\{\mathcal L(x,\hat y)-\mathcal L(\hat x,y)-\frac{\xi}{2}\|\hat x-x^*\|^2-\frac{\xi}{2}\|\hat y-y^*\|^2\right\} \\ 
        & \ =\frac{1}{2}x^TQx+c^Tx+b^Ty+\iota_{\mathbb R^m_+}(y)\\
        & \ \quad + \max_{\hat y}\ \left\{\hat y^TAx-b^T\hat y-\iota_{\mathbb R_+^m}(\hat y)-\frac{\xi}{2}\|\hat y-y^*\|^2\right\} \\ 
        & \ \quad + \max_{\hat x}\ \left\{-y^TA\hat x-\frac{1}{2}\hat x^TQ\hat x-c^T\hat x-\frac{\xi}{2}\|\hat x-x^*\|^2\right\} \ .
    \end{split}    
    \end{align}
    }
    
    Therefore we need to calculate the last two terms in \eqref{eq:gap-rewrite}. Note that
    {\small
    \begin{align}\label{eq:gap-term-1}
        \begin{split}
            \max_{\hat y}\left\{ \hat y^TAx-b^T\hat y-\iota_{\mathbb R_+^m}(\hat y)-\frac{\xi}{2}\|\hat y-y^*\|^2\right\} & \ = \max_{\hat y\geq 0} \left\{\hat y^TAx-b^T\hat y-\frac{\xi}{2}\|\hat y-y^*\|^2 \right\} \\
            & \ = \max_{\hat y\geq 0} \left\{-\frac{\xi}{2}\left\| \hat y-\pran{y^*+\frac{1}{\xi}(Ax-b)} \right\|^2\right\}+\frac{\xi}{2}\left\|y^*+\frac{1}{\xi}(Ax-b)\right\|^2-\frac{\xi}{2}\|y^*\|^2\\
            & \ =\frac{\xi}{2}\left\| \left[ y^*+\frac{1}{\xi}(Ax-b)\right]_+ \right\|^2-\frac{\xi}{2}\|y^*\|^2 \ ,
        \end{split}
    \end{align}
    }
    and
    {\small
    \begin{align}\label{eq:gap-term-2}
        \begin{split}
            \max_{\hat x} \left\{-y^TA\hat x-\frac{1}{2}\hat x^TQ\hat x-c^T\hat x-\frac{\xi}{2}\|\hat x-x^*\|^2\right\} & \ = \max_{\hat x} \left\{-\frac{\xi}{2}\left\|\hat x-\pran{I+\frac{1}{\xi}Q}^{-1}\pran{x^*-\frac{1}{\xi}(c+A^Ty)}\right\|_{I+\frac{1}{\xi}Q}^2\right\}\\
            & \ \quad +\frac{\xi}{2}\left\| x^*-\frac{1}{\xi}(c+A^Ty)\right\|_{\pran{I+\frac{1}{\xi}Q}^{-1}}^2-\frac{\xi}{2}\|x^*\|^2\\
            & \ =\frac{\xi}{2}\left\| x^*-\frac{1}{\xi}(c+A^Ty)\right\|_{\pran{I+\frac{1}{\xi}Q}^{-1}}^2-\frac{\xi}{2}\|x^*\|^2 \ .
        \end{split}
    \end{align}
    } 
    Combining \eqref{eq:gap-rewrite}, \eqref{eq:gap-term-1} and \eqref{eq:gap-term-2}, we achieve \eqref{eq:gap-explicit}.
\end{proof}

The following simple lemma implies that we can split the (primal-dual) smoothed duality gap into the summation of the primal and dual parts, either of which has a simplified form.
\begin{lem}\label{lem:pd=p-d}
Let $z^*=(x^*,y^*)$ be an optimal solution to \eqref{eq:minimax} and denote $z=(x,y)\in \mathcal Z$. Then for any $\xi>0$,

(i) The smoothed duality gap at $(x,y)$ can be written as
\begin{equation*}
        G_{\xi}((x,y);z^*) = G_{\xi}((x,y^*);z^*)+G_{\xi}((x^*,y);z^*) \ ,
    \end{equation*}

(ii) The smoothed duality gap at $(x,y^*)$ equals
\begin{equation}\label{eq:gap-primal}
        G_{\xi}((x,y^*);z^*) = \frac{\xi}{2}\left\| \left[y^*+\frac{1}{\xi}(Ax-b)\right]_+ \right\|^2-\frac{\xi}{2}\|y^*\|^2-{y^*}^T(Ax-b)+\frac{1}{2}\|x-x^*\|_Q^2 \ ,
    \end{equation}

(iii) The smoothed duality gap at $(x^*,y)$ equals
\begin{equation}\label{eq:gap-dual}
        G_{\xi}((x^*,y);z^*) = \iota_{\mathbb R^m_+}(y)+\frac{1}{2\xi}\|Qx^*+c+A^Ty\|_{\pran{I+\frac{1}{\xi}Q}^{-1}}^2+y^T(b-Ax^*) \ .
\end{equation}
\end{lem}
\begin{proof}
(i) The proof follows from the definition of the smoothed duality gap.
{\footnotesize
    \begin{align*}
        \begin{split}
            G_{\xi}((x,y);z^*) & = \max_{\hat x,\hat y}\left\{\mathcal L(x,\hat y)-\mathcal L(\hat x,y)-\frac{\xi}{2}\| \hat x-x^*\|^2-\frac{\xi}{2}\| \hat y-y^*\|^2\right\} \\
            & =\max_{\hat x,\hat y}\left\{\mathcal L(x,\hat y)-\mathcal L(x^*,y^*) + \mathcal L(x^*,y^*)-\mathcal L(\hat x,y)-\frac{\xi}{2}\| \hat x-x^*\|^2-\frac{\xi}{2}\| \hat y-y^*\|^2 \right\}\\
            & = \max_{\hat y} \left\{\mathcal L(x,\hat y)-\mathcal L(x^*,y^*)-\frac{\xi}{2}\| \hat y-y^*\|^2\right\} + \max_{\hat x} \left\{\mathcal L(x^*,y^*)-\mathcal L(\hat x,y)-\frac{\xi}{2}\| \hat x-x^*\|^2 \right\}\\
            & =\max_{\tilde x, \hat y} \left\{\mathcal L(x,\hat y)-\mathcal L(\tilde x,y^*)-\frac{\xi}{2}\| \hat y-y^*\|^2 -\frac{\xi}{2}\| \tilde x-x^*\|^2\right\} + \max_{\hat x,\tilde y} \left\{\mathcal L(x^*,\tilde y)-\mathcal L(\hat x,y)-\frac{\xi}{2}\| \hat x-x^*\|^2-\frac{\xi}{2}\|\tilde y-y^*\|^2\right\}\\
            & =G_{\xi}((x,y^*);z^*)+G_{\xi}((x^*,y);z^*) \ .
        \end{split}
    \end{align*}
}

(ii) Plugging $(x,y^*)$ into \eqref{eq:gap-explicit}, we have
\begin{align}\label{eq:gap-p}
        \begin{split}
            G_{\xi}((x,y^*);z^*)& \ = \frac{1}{2}x^TQx+c^Tx+b^Ty^* \\
            & \ \quad+\frac{\xi}{2}\left\| \left[y^*+\frac{1}{\xi}(Ax-b)\right]_+ \right\|^2-\frac{\xi}{2}\|y^*\|^2\\
            & \ \quad+\frac{\xi}{2}\left\| x^*-\frac{1}{\xi}(c+A^Ty^*)\right\|_{\pran{I+\frac{1}{\xi}Q}^{-1}}^2-\frac{\xi}{2}\|x^*\|^2  \ .
        \end{split}
\end{align}
The last term in \eqref{eq:gap-p} can be further simplified as follows
    \begin{align*}
        \begin{split}
            \frac{\xi}{2}\left\| x^*-\frac{1}{\xi}(c+A^Ty^*)\right\|_{\pran{I+\frac{1}{\xi}Q}^{-1}}^2-\frac{\xi}{2}\|x^*\|^2 & \ = \frac{\xi}{2}\left\| x^*-\frac{1}{\xi}(-Qx^*)\right\|_{\pran{I+\frac{1}{\xi}Q}^{-1}}^2-\frac{\xi}{2}\|x^*\|^2\\
            & \ = \frac{\xi}{2}\left\| x^*\right\|_{\pran{I+\frac{1}{\xi}Q}}^2 -\frac{\xi}{2}\|x^*\|^2\\
            & \ =\frac{1}{2}{x^*}^{T}Qx^* \ .
        \end{split}
    \end{align*}
    Accordingly, it holds that
    {\small
    \begin{align*}
        \begin{split}
            G_{\xi}((x,y^*);z^*)& \ = \frac{1}{2}x^TQx+\frac{1}{2}{x^*}^{T}Qx^*+c^Tx+b^Ty^* +\frac{\xi}{2}\left\| \left[y^*+\frac{1}{\xi}(Ax-b)\right]_+ \right\|^2-\frac{\xi}{2}\|y^*\|^2\\
            & \ = \frac{1}{2}x^TQx+\frac{1}{2}{x^*}^{T}Qx^*+(-Qx^*-A^Ty^*)^Tx+b^Ty^* +\frac{\xi}{2}\left\| \left[y^*+\frac{1}{\xi}(Ax-b)\right]_+ \right\|^2-\frac{\xi}{2}\|y^*\|^2\\
            & \ = \frac{1}{2}\|x-x^*\|_Q^2 + \frac{\xi}{2}\left\| \left[y^*+\frac{1}{\xi}(Ax-b)\right]_+ \right\|^2-\frac{\xi}{2}\|y^*\|^2+{y^*}^T(b-Ax) \ .
        \end{split}
    \end{align*}
    }
    
(iii)  Plugging $(x,y^*)$ into \eqref{eq:gap-explicit}, we have 
\begin{align}\label{eq:gap-d}
    \begin{split}
        G_{\xi}((x^*,y);z^*)& \ =\frac{1}{2}{x^*}^TQx^*+c^Tx^*+b^Ty+\iota_{\mathbb R^m_+}(y) \\
        & \ \quad+\frac{\xi}{2}\left\| \left[y^*+\frac{1}{\xi}(Ax^*-b)\right]_+ \right\|^2-\frac{\xi}{2}\|y^*\|^2\\
        & \ \quad+\frac{\xi}{2}\left\| x^*-\frac{1}{\xi}(c+A^Ty)\right\|_{\pran{I+\frac{1}{\xi}Q}^{-1}}^2-\frac{\xi}{2}\|x^*\|^2 \ .
    \end{split}    
\end{align}
    Note that due to complementary slackness,
    \begin{equation*}
        \frac{\xi}{2}\left\| \left[y^*+\frac{1}{\xi}(Ax^*-b)\right]_+ \right\|^2-\frac{\xi}{2}\|y^*\|^2 = 0 \ ,
    \end{equation*}
    and the last term in \eqref{eq:gap-d} equals
    \begin{align*}
    \begin{split}
        \frac{\xi}{2}\left\| x^*-\frac{1}{\xi}(c+A^Ty)\right\|_{\pran{I+\frac{1}{\xi}Q}^{-1}}^2-\frac{\xi}{2}\|x^*\|^2 & \ = \frac{\xi}{2}\left\| x^*-\frac{1}{\xi}(c+A^Ty+Qx^*)+\frac{1}{\xi}Qx^*\right\|_{\pran{I+\frac{1}{\xi}Q}^{-1}}^2-\frac{\xi}{2}\|x^*\|^2 \\
        & \ = \frac{1}{2}{x^*}^TQx^*+\frac{1}{2\xi}\|c+A^Ty+Qx^*\|_{\pran{I+\frac{1}{\xi}Q}^{-1}}^2-{x^*}^T(c+A^Ty+Qx^*) \ .
    \end{split}    
    \end{align*}
    Thus we achieve
    \begin{align*}
    \begin{split}
        G_{\xi}((x^*,y);z^*)& \ =\frac{1}{2}{x^*}^TQx^*+c^Tx^*+b^Ty+\iota_{\mathbb R^m_+}(y) \\
        & \ \quad+\frac{1}{2}{x^*}^TQx^*+\frac{1}{2\xi}\|c+A^Ty+Qx^*\|_{\pran{I+\frac{1}{\xi}Q}^{-1}}^2-{x^*}^T(c+A^Ty+Qx^*)\\
        & \ = \frac{1}{2\xi}\|c+A^Ty+Qx^*\|_{\pran{I+\frac{1}{\xi}Q}^{-1}}^2+y^T(b-Ax^*)+\iota_{\mathbb R^m_+}(y) \ .
    \end{split}    
    \end{align*}
\end{proof}

An auxiliary lemma, as stated below, is helpful to build the bound to KKT residual by the smoothed duality gap. 
\begin{lem}\label{lem:residual-positive}
Let $z^*=(x^*,y^*)$ be an optimal solution to \eqref{eq:minimax} and denote $z=(x,y)\in \mathcal Z$. Let $\xi>0$.
\begin{enumerate}
    \item[(i)] Denote $\mathcal I_+=\{i:y_i^*+\frac{1}{\xi}(A_ix-b_i)\geq 0  \}$ and $I_-=\{i:y_i^*+\frac{1}{\xi}(A_ix-b_i)< 0  \}$.  Then
    \begin{equation*}
        \sum_{i\in \mathcal I_-}\frac{1}{2}y_i^*(b_i-A_ix)\geq 0 \ ,
    \end{equation*}
    and
    \begin{equation*}
        0\leq  \sum_{i\in \mathcal I_+} \frac{1}{2\xi} (A_ix-b_i)^2 + \sum_{i\in \mathcal I_-}\frac{1}{2}y_i^*(b_i-A_ix)+\frac{1}{2}\|x-x^*\|_Q^2 \leq G_{\xi}((x,y);z^*) \ .
    \end{equation*}
    \item[(ii)] It holds that
    \begin{equation*}
        y^T(b-Ax^*)+\iota_{\mathbb R^m_+}(y) \geq 0 \ ,
    \end{equation*}
    and
    \begin{equation*}
        0\leq \frac{1}{2\xi}\|c+A^Ty+Qx^*\|_{\pran{I+\frac{1}{\xi}Q}^{-1}}^2+y^T(b-Ax^*)+\iota_{\mathbb R^m_+}(y) \leq G_{\xi}((x,y);z^*) \ .
    \end{equation*}
\end{enumerate}  
\end{lem}
\begin{proof}
    (i) Recall Equation \eqref{eq:gap-primal}. For $i\in\mathcal I_+=\{i:y_i^*+\frac{1}{\xi}(A_ix-b_i)\geq 0  \}$,  we have
    {\small
    \begin{align}\label{eq:eq-p-1}
        \begin{split}
            \frac{\xi}{2}\pran{\left[y_i^*+\frac{1}{\xi}(A_ix-b_i)\right]_+ }^2-\frac{\xi}{2}(y_i^*)^2-{y_i^*}(A_ix-b_i) & = \frac{\xi}{2}\pran{y_i^*+\frac{1}{\xi}(A_ix-b_i) }^2-\frac{\xi}{2}(y_i^*)^2-{y_i^*}(A_ix-b_i)\\
            & = \frac{1}{2\xi}\pran{A_ix-b_i}^2 \geq 0 \ .
        \end{split}
    \end{align}
    }
    For $i\in\mathcal I_-=\{i:y_i^*+\frac{1}{\xi}(A_ix-b_i)< 0  \}$, it holds that $A_ix-b_i<-\xi y_i^*\leq 0$ and 
    {\small
    \begin{align}\label{eq:eq-p-2}
        \begin{split}
            \frac{\xi}{2}\pran{\left[y_i^*+\frac{1}{\xi}(A_ix-b_i)\right]_+ }^2-\frac{\xi}{2}(y_i^*)^2-{y_i^*}(A_ix-b_i) & = -\frac{\xi}{2}(y_i^*)^2-{y_i^*}(A_ix-b_i)\\
            & = y_i^*\pran{b_i-A_ix-\frac{\xi}{2}y_i^*} \\
            & \geq \frac 12 y_i^*(b_i-A_ix) \geq \frac{\xi}{2} {y_i^*}^2  \geq 0 \ .
        \end{split}
    \end{align}
    }
    Thus
    \begin{align*}
        \begin{split}
            G_{\xi}((x,y);z^*) & \geq G_{\xi}((x,y^*);z^*) \\ 
            & = \frac{\xi}{2}\left\| \left[y^*+\frac{1}{\xi}(Ax-b)\right]_+ \right\|^2-\frac{\xi}{2}\|y^*\|^2-{y^*}^T(Ax-b)+\frac{1}{2}\|x-x^*\|_Q^2 \\ 
            & \geq \sum_{i\in \mathcal I_+} \frac{1}{2\xi} (A_ix-b_i)^2 + \sum_{i\in \mathcal I_-}\frac{1}{2}y_i^*(b_i-A_ix)+\frac{1}{2}\|x-x^*\|_Q^2 \geq 0 \ ,
        \end{split}
    \end{align*}
    where the first inequality and equality follow from part (i) and (ii) in Lemma \ref{lem:pd=p-d}, respectively, and the last inequality combines \eqref{eq:eq-p-1} and \eqref{eq:eq-p-2}.

    (ii) Since $b-Ax^*\geq 0$, we have for any $y\geq 0$
    \begin{equation*}
        y^T(b-Ax^*)+\iota_{\mathbb R^m_+}(y) \geq 0 \ ,
    \end{equation*}
    and thus 
    \begin{equation*}
        G_{\xi}((x,y);z^*) \geq G_{\xi}((x^*,y);z^*) = \frac{1}{2\xi}\|Qx^*+c+A^Ty\|_{\pran{I+\frac{1}{\xi}Q}^{-1}}^2+y^T(b-Ax^*)+\iota_{\mathbb R^m_+}(y) \geq 0 \ .
    \end{equation*}
\end{proof}
The next result is the key lemma to the proof of Theorem \ref{thm:qp-qg}, which implies that we can indeed upper bound each residual in KKT system (i.e., primal residual, dual residual, and complementarity) by smoothed duality gap.
\begin{lem}\label{lem:residual}
Let $z^*=(x^*,y^*)$ be an optimal solution to \eqref{eq:minimax} and denote $z=(x,y)\in \mathcal Z$. Then it holds for $\xi>0$ that
    \begin{enumerate}
        \item[(i)] Primal residual: \begin{equation*}
        \| [Ax-b]_+ \|^2 \leq {2\xi G_{\xi}(z;z^*)} \ ,
    \end{equation*}
     \item[(ii)] Dual residual:
     \begin{equation*}
         \| Qx+c+A^Ty^* \|^2 \leq {2\|Q\|G_{\xi}(z;z^*)} \quad \text{and} \quad \|Qx^*+c+A^Ty\|_{\pran{I+\frac{1}{\xi}Q}^{-1}}^2\leq {2\xi G_{\xi}(z;z^*)} \ ,
     \end{equation*}
     \item[(iii)] Complementarity:
     \begin{equation*}
        \left\|y-\left[y-\frac{1}{\xi}(b-Ax^*)\right]_+\right\|^2 \leq \frac{1}{\xi}G_{\xi}(z;z^*) \quad \text{and} \quad \left\|y^*-\left[y^*-\frac{1}{\xi}(b-Ax)\right]_+\right\|^2 \leq \frac{2}{\xi}G_{\xi}(z;z^*) \ .
    \end{equation*}
    \end{enumerate}
\end{lem}
\begin{proof}
    (i) Recall the definition of $\mathcal I_+=\{i:y_i^*+\frac{1}{\xi}(A_ix-b_i)\geq 0  \}$.
    \begin{equation*}
        \| [Ax-b]_+ \|^2 \leq \sum_{i\in \mathcal I_+} (A_ix-b_i)^2 \leq 2\xi G_{\xi}((x,y^*);z^*) \leq 2\xi G_{\xi}(z;z^*) \ ,
    \end{equation*}
    where the second inequality uses part (i) of Lemma \ref{lem:residual-positive}.

    (ii) By part (i) of Lemma \ref{lem:residual-positive}, we have $\frac 12\|x-x^*\|_Q^2\leq G_{\xi}((x,y^*);z^*)$ and thus
    \begin{equation*}
        \| Qx+c+A^Ty^* \|^2 = \|Qx-Qx^*\|^2 \leq \|Q\|\|x-x^*\|_Q^2 \leq {2\|Q\|G_{\xi}(z;z^*)} \ .
    \end{equation*}
    Moreover, we have
    \begin{equation*}
        \frac{1}{2\xi}\|Qx^*+c+A^Ty\|_{\pran{I+\frac{1}{\xi}Q}^{-1}}^2\leq G_{\xi}((x^*,y);z^*)\leq  G_{\xi}(z;z^*) \ ,
    \end{equation*}
    where the first inequality is due to part (ii) of Lemma \ref{lem:residual-positive}.

    (iii) Denote $\mathcal J_1:=\{j:b_j-A_jx^*\leq \xi y_j\}$ and $\mathcal J_2:=\{j:b_j-A_jx^*> \xi y_j\}$. Then for any $y\geq 0$,
    \begin{align*}
        \begin{split}
            \left\|y-\left[y-\frac{1}{\xi}(b-Ax^*)\right]_+\right\|^2 & \ =\frac{1}{\xi^2}\sum_{j\in \mathcal J_1} (b_j-A_jx^*)^2+\sum_{j\in \mathcal J_2}y_j^2 \\
            & \ \leq \frac{1}{\xi}\sum_{j\in \mathcal J_1} y_j(b_j-A_jx^*)+\frac{1}{\xi}\sum_{j\in \mathcal J_2}y_j(b_j-A_jx^*) \\
            & = \frac{1}{\xi}y^T(b-Ax^*) \leq \frac{1}{\xi}G_{\xi}((x,y);z^*) \ ,
        \end{split}
    \end{align*}
    where the first inequality is from the definition of $\mathcal J_1$ and $\mathcal J_2$, and the second inequality uses part (ii) of Lemma \ref{lem:residual-positive}.
    This finishes the proof of the first inequality.
    
    To show the second one, recall $\mathcal I_+=\{i:y_i^*+\frac{1}{\xi}(A_ix-b_i)\geq 0  \}$ and $\mathcal I_-=\{i:y_i^*+\frac{1}{\xi}(A_ix-b_i)< 0\}$ 
    \begin{align*}
        \begin{split}
            \left\|y^*-\left[y^*-\frac{1}{\xi}(b-Ax)\right]_+\right\|^2 & \ =\sum_{i\in\mathcal I_+}\left(y_i^*-\left[y_i^*-\frac{1}{\xi}(b_i-A_ix)\right]_+\right)^2+\sum_{i\in\mathcal I_-}\left(y_i^*-\left[y_i^*-\frac{1}{\xi}(b_i-A_ix)\right]_+\right)^2 \\
            & \ = \frac{1}{\xi^2}\sum_{i\in\mathcal I_+}(b_i-A_ix)^2+\sum_{i\in\mathcal I_-}{y_i^*}^2\\
            & \ \leq \frac{1}{\xi^2}\sum_{i\in\mathcal I_+}(b_i-A_ix)^2+\frac{1}{\xi}\sum_{i\in\mathcal I_-}y_i^*(b_i-A_ix)\\
            & \leq \frac{2}{\xi} G_{\xi}((x,y);z^*) \ ,
        \end{split}
    \end{align*}
    where the first equality follows the definition of $\mathcal I_+$ and $\mathcal I_-$, and the last inequality is due to part (i) of Lemma \ref{lem:residual-positive}.
\end{proof}
Now we are ready to prove Theorem \ref{thm:qp-qg}.
\begin{proof}[Proof of Theorem \ref{thm:qp-qg}]
By definition of $\gamma_\xi$ in \eqref{eq:gamma}, we have
\begin{equation}\label{eq:gamma-pd}
    \begin{aligned}
        & \gamma_\xi^2\mathrm{dist}^2(x,\mathcal X^*) \leq \|[Ax-b]_+\|^2 +\| Qx+c+A^Ty^*\|_{\pran{I+\frac{1}{\xi}Q}^{-1}}^2+\left\|y^*-\left[y^*-\frac{1}{\xi}(b-Ax)\right]_+\right\|^2 \\ 
        & \qquad\qquad\qquad \  \leq \|[Ax-b]_+\|^2 +\| Qx+c+A^Ty^*\|^2+\left\|y^*-\left[y^*-\frac{1}{\xi}(b-Ax)\right]_+\right\|^2 \\
        & \gamma_\xi^2\mathrm{dist}^2(y,\mathcal Y^*)\leq \| Qx^*+c+A^Ty\|_{\pran{I+\frac{1}{\xi}Q}^{-1}}^2+\left\|y-\left[y-\frac{1}{\xi}(b-Ax^*)\right]_+\right\|^2 \ .
    \end{aligned}
\end{equation}

Combining \eqref{eq:gamma-pd} and Lemma \ref{lem:residual}, we have
    \begin{align*}
        \begin{split}
            \gamma_\xi^2\mathrm{dist}^2(z,\mathcal Z^*)& = \gamma_\xi^2\mathrm{dist}^2(x,\mathcal X^*)+\gamma_\xi^2\mathrm{dist}^2(y,\mathcal Y^*)\\
            & \leq \|[Ax-b]_+\|^2 +\| Qx+c+A^Ty^*\|^2+\left\|y^*-\left[y^*-\frac{1}{\xi}(b-Ax)\right]_+\right\|^2\\
            & \quad + \| Qx^*+c+A^Ty\|_{\pran{I+\frac{1}{\xi}Q}^{-1}}^2+\left\|y-\left[y-\frac{1}{\xi}(b-Ax^*)\right]_+\right\|^2\\
            & \leq (2\xi+2\|Q\|+2/\xi+2\xi+1/\xi) G_{\xi}((x,y);z^*)=(4\xi+2\|Q\|+3/\xi)G_{\xi}((x,y);z^*) \ .
        \end{split}
    \end{align*}
    Divide both sides by $4\xi+2\|Q\|+3/\xi$, we reach that for any $z\in \mathcal Z$ with $\|z\|\leq R$,
    \begin{equation*}
        \alpha_\xi\mathrm{dist}^2(z,\mathcal Z^*)= \frac{\gamma_\xi^2}{4\xi+2\|Q\|+3/\xi}\mathrm{dist}^2(z,\mathcal Z^*)\leq G_{\xi}((x,y);z^*) \ .
    \end{equation*}
\end{proof}

\section{Lower Bound for Primal-Dual FOMs}\label{sec:lower_bound}

This section presents a lower bound of primal-dual FOMs for solving QP in terms of the quadratic growth constant $\alpha_\xi$. Together with Theorem \ref{thm:main}, it demonstrates that rAPDHG (Algorithm \ref{alg:rapd}) achieves the best possible worst-case convergence bound for a large class of primal-dual problems for QP.

Similar to that in \cite{applegate2023faster}, we consider the span-respecting algorithms in the primal-dual setting. This class of algorithms includes many primal-dual algorithms such as primal-dual hybrid gradient, extra-gradient method, gradient descent-ascent, etc.
\begin{mydef}[{\cite[Definition 3]{applegate2023faster}}]\label{def:def-span}
    A primal-dual algorithm is called first-order span-respecting for an unconstrained problem $\min_{x\in\mathbb R^n}\max_{y\in \mathbb R^m}\mathcal L(x,y)$ if its iterates satisfy
    \begin{align*}
        \begin{split}
            & \ x^k\in x^0+\mathrm{Span}\left\{\nabla_x\mathcal L(x^i,y^j):\;\forall i,j\in\{0,...,k-1\}\right\}\\
            & \ y^k \in y^0+\mathrm{Span}\left\{\nabla_y\mathcal L(x^i,y^j):\;\forall i\in\{0,...,k\},\forall j\in\{0,...,k-1\}\right\} \ .
        \end{split}
    \end{align*}
\end{mydef}

We utilize the classic lower-bound example of Nemirovski/Nesterov for strongly convex functions to create our worst-case QP instance for primal-dual first-order methods. In particular, we know that:
\begin{lem}[{\cite[Theorem 2.1.12]{nesterov2003introductory}}]\label{lem:lb}
    For all $L> \mu>0$, there exists a positive definite matrix $H\in \mathbb R^{n\times n}$ and vector $h\in\mathbb R^n$ such that $\|H\|=L$ and $\sigma_{\mathrm{min}}(H)=\mu$ and solving
    \begin{equation}\label{eq:lb-instance}
        \min_{x\in\mathbb R^n}\ f(x):=\frac 12 x^THx+h^Tx \ ,
    \end{equation}
    with any span-respecting minimization algorithm $w^k\in w^0+\mathrm{Span}\left\{\nabla f(w^i):\;\forall i\in\{0,...,k-1\}\right\}$ satisfies for $k<n$ that
    \begin{equation*}
        \mathrm{dist}(w^k,\mathcal W^*)\geq \pran{1-\sqrt\frac{\mu}{L}}^k\mathrm{dist}(w^0,\mathcal W^*) \ ,
    \end{equation*}
    {where $\mathcal W^*$ is the optimal solution set of \eqref{eq:lb-instance},} namely, it requires at least 
    \begin{equation*}
        \Omega\pran{\sqrt\frac{L}{\mu}\log\frac{1}{\epsilon}}=\Omega\pran{\sqrt\frac{\|H\|}{\sigma_{\mathrm{min}}(H)}\log\frac{1}{\epsilon}}
    \end{equation*}
    iterations to achieve $\epsilon$-accuracy {in the sense that distance to optimality is no larger than $\epsilon$}.
\end{lem}

The lower complexity result for convex QP is established in the following theorem.
\begin{thm}\label{thm:lb}
Consider span-respecting algorithms for solving quadratic programs \eqref{eq:minimax}.

(i) There exists a QP instance such that any span-respecting algorithm requires at least
    \begin{equation}\label{eq:lb-bg}
        \Omega\left(\min_{\xi>0}\left\{\max\left\{\frac{\|A\|}{\alpha_{\xi}},\frac{\|A\|}{\xi} \right\} \right\} \log\frac{1}{\epsilon}\right)
    \end{equation}
    iterations to find an $\epsilon$-accurate solution {in the sense that distance to optimality is no larger than $\epsilon$}.

(ii) There exists a QP instance such that any span-respecting algorithm  requires at least
    \begin{equation}\label{eq:lb-sc}
        \Omega\left(\min_{\xi>0}\left\{\max\left\{ \sqrt{\frac{\|Q\|}{\alpha_{\xi}}}, \sqrt{\frac{\|Q\|}{\xi}} \right\} \right\} \log\frac{1}{\epsilon}\right)
    \end{equation}
    iterations to find an $\epsilon$-accurate solution {in the sense that distance to optimality is no larger than $\epsilon$}.    
\end{thm}

\begin{proof}

    (i) 
    We prove it by constructing a worst-case instance based on bilinear games with the following QP form:
    \begin{equation}\label{eq:minimax-bg}
    \min_x\max_{y}\; \mathcal L(x,y):= y^TAx+b^Ty \ ,
    \end{equation}
    where $A=H^{1/2}$, $b=H^{-1/2}h$, and $H$ and $h$ are specified in the worst-case instance in Lemma \ref{lem:lb}.

    In Corollary 1 of \cite{applegate2023faster}, it was shown that any span-respecting primal-dual algorithm for solving \eqref{eq:minimax-bg} requires at least
    \begin{equation}\label{eq:lb-bg-2}
        \Omega\pran{\frac{\|A\|}{\sigma_{\mathrm{min}}^+(A)}\log\frac{1}{\epsilon}}
    \end{equation}
    iterations to achieve $\epsilon$-accuracy. Below we will prove \eqref{eq:lb-bg} matches \eqref{eq:lb-bg-2} on bilinear games \eqref{eq:minimax-bg} to conclude the proof of part (i).

We characterize the growth constant $\alpha_\xi$ for bilinear games \eqref{eq:minimax-bg}. The growth constant $\alpha_\xi$ of \eqref{eq:minimax-bg} for any $\xi>0$ is controlled by the minimum positive singular value of matrix $A$. Suppose a finite solution to \eqref{eq:minimax-bg} exists. Then by the definition of smoothed duality gap and straightforward calculation, we have
    \begin{align*}
        \begin{split}
            G_{\xi}(z;z^*) & \ = \max_{\hat z}\left\{ \hat y^TAx+b^T\hat y-y^TA\hat x-b^Ty-\frac{\xi}{2}\|\hat x-x^*\|^2-\frac{\xi}{2}\|\hat y-y^*\|^2\right\} \\ 
            & \ = \max_{\hat y} \left\{\hat y^TAx+b^T\hat y-\frac{\xi}{2}\|\hat y-y^*\|^2 \right\}+\max_{\hat x} \left\{y^TA\hat x-b^Ty-\frac{\xi}{2}\|\hat x-x^*\|^2\right\}\\ 
            & \ = \frac{1}{2\xi}\left\|\begin{pmatrix}
                0 & A^T\\A & 0
            \end{pmatrix}\begin{pmatrix}
                x-x^* \\ y-y^*
            \end{pmatrix}\right\|^2 \\
            & \ \geq \frac{\sigma_{\mathrm{min}}^+(A)^2}{2\xi}\mathrm{dist}^2(z,\mathcal Z^*) \ ,
        \end{split}
    \end{align*}
    where the third equation utilizes $A^Ty^*=0$ and $Ax^*=b$ and inequality follows from the definition of $\sigma_{\mathrm{min}}^+(A)$.
Thus the smoothed duality gap of \eqref{eq:minimax-bg} satisfies quadratic growth on $\mathcal Z$ with $\alpha_\xi=\frac{\sigma_{\mathrm{min}}^+(A)^2}{2\xi}$, i.e., for all $z^*\in\mathcal Z^*$ and any $z\in\mathcal Z$,
    \begin{equation*}
        G_{\xi}(z;z^*) \geq \frac{\sigma_{\mathrm{min}}^+(A)^2}{2\xi} \mathrm{dist}^2(z,\mathcal Z^*) \ .
    \end{equation*}
Thus it holds that
\begin{equation*}
    \Omega\pran{\min_{\xi>0}\left\{\max\left\{\frac{\|A\|}{\alpha_\xi},\frac{\|A\|}{\xi}\right\}\right\}\log\frac{1}{\epsilon}}=\Omega\pran{\frac{\|A\|}{\sigma_{\mathrm{min}}^+(A)}\log\frac{1}{\epsilon}} \ ,
\end{equation*}
which finishes the proof by using Corollary 1 of \cite{applegate2023faster}.

(ii) We prove by constructing a worst-case instance based on strongly convex quadratic program with the following form.
\begin{align}\label{eq:minimax-sc}
\begin{split}
    \min_{x=(\tilde x,x_0)\in \mathbb R^{n+1}}\max_{y=(\tilde y,y_0)\in\mathbb R^{n+1}} \frac 12 x^TQx+c^Tx+y^TAx \ ,
\end{split}
\end{align}
where $Q=\begin{pmatrix}
    H & 0_{n\times 1} \\ 0_{1\times n} & 0
\end{pmatrix}\in \mathbb R^{(n+1)\times (n+1)}$, $c=\begin{pmatrix}
    h \\ 0
\end{pmatrix}\in \mathbb R^{n+1}$, $A=\begin{pmatrix}
    0_{n\times n} & 0_{n\times 1}\\0_{1\times n} & \sqrt{\|Q\|\mu}
\end{pmatrix}\in \mathbb R^{(n+1)\times (n+1)}$. 

Note that essentially \eqref{eq:minimax-sc} is equivalent to the following unconstrained quadratic minimization problem
\begin{align}\label{eq:min-sc}
\begin{split}
    \min_{\tilde x\in \mathbb R^{n}}\; f(x):=\frac 12 \tilde x^TH\tilde x+h^T\tilde x \ ,
\end{split}
\end{align}
and any primal-dual span-respecting algorithm has
\begin{equation*}
    \tilde x^k\in \tilde x^0+\mathrm{Span}\left\{H\tilde x^0+h,...,H\tilde x^{k-1}+h\right\}=x^0+\mathrm{Span}\left\{\nabla f(\tilde x^0),...,\nabla f(\tilde x^{k-1})\right\}
\end{equation*}

Thus $\tilde x^k$ can be viewed as iterates of span-respecting algorithms for solving \eqref{eq:min-sc} and it is shown in Lemma \ref{lem:lb} that any span-respecting algorithm for solving \eqref{eq:min-sc} requires at least 
\begin{equation}\label{eq:lb-sc-2}
    \Omega\pran{\sqrt{\frac{\|H\|}{\sigma_{\mathrm{min}}(H)}}\log\frac{1}{\epsilon}}
\end{equation}
iterations to achieve $\epsilon$-accuracy. Below we will prove \eqref{eq:lb-sc} on QP \eqref{eq:minimax-sc} matches \eqref{eq:lb-sc-2} to finish the proof of part (ii).

We characterize the growth constant $\alpha_\xi$ is characterized by strong convexity $\mu$.
Let $\xi>0$. It holds by the definition of smoothed duality gap that
\begin{align*}
    G_{\xi}(z;z^*)& = \frac{1}{2}x^TQx+c^Tx+\max_{\hat y}\left\{{\hat y}^TAx-\frac{\xi}{2}\|\hat y-y^*\|^2\right\} + \max_{\hat x} \left\{-\frac{1}{2}{\hat x}^TQ\hat x-c^T\hat x-y^TA\hat x-\frac{\xi}{2}\|\hat x-x^*\|^2\right\}\\
    & = \frac{1}{2}\|x-x^*\|_Q^2+\frac{1}{2\xi}\|Ax\|^2+\frac{1}{2\xi}\|c+A^Ty+Qx^*\|_{\pran{I+\frac{1}{\xi}Q}^{-1}}^2-y^TAx^*\\
    & = \frac{1}{2}\|\tilde x-\tilde x^*\|_{\widetilde Q}^2+\frac{\|Q\|\mu}{2\xi}x_0^2+\frac{\|Q\|\mu}{2\xi}y_0^2\geq \frac{1}{2}\|\tilde x-\tilde x^*\|_{\widetilde Q}^2+\frac{\mu^2}{2\xi}x_0^2+\frac{\mu^2}{2\xi}y_0^2\\
    & \geq \min\left\{\frac{\mu^2}{2\xi},\frac{\mu}{2}\right\}(\mathrm{dist}^2(x,\mathcal X^*)+\mathrm{dist}^2(y,\mathcal Y^*))=\min\left\{\frac{\mu^2}{2\xi},\frac{\mu}{2}\right\} \mathrm{dist}^2(z,\mathcal Z^*) \ ,
\end{align*}
where the second equality is from direct calculation, the first inequality follows from $\|Q\|\geq \mu$ and the last inequality uses $x_0^*=0$ and $\mathcal Y^*=\mathbb R^n\times\{0\}$. This implies the smoothed duality gap of \eqref{eq:minimax-sc} satisfies quadratic growth on $\mathcal Z$ with $\alpha_\xi=\min\left\{\frac{\mu^2}{2\xi},\frac{\mu}{2}\right\}$, that is,
    \begin{equation*}
        G_{\xi}(z;z^*) \geq \min\left\{\frac{\mu^2}{2\xi},\frac{\mu}{2}\right\} \mathrm{dist}^2(z,\mathcal Z^*) \ .
    \end{equation*}
Thus complexity \eqref{eq:lb-sc} on QP \eqref{eq:minimax-sc} equals rate \eqref{eq:lb-sc-2} 
\begin{equation*}
    \Omega\pran{\min_{\xi>0}\left\{\max\left\{\sqrt{\frac{\|Q\|}{\alpha_{\xi}}},\sqrt{\frac{\|Q\|}{\xi}}\right\}\right\}\log\frac{1}{\epsilon}}=\Omega\pran{\sqrt{\frac{\|Q\|}{\mu}}\log\frac{1}{\epsilon}}=\Omega\pran{\sqrt{\frac{\|H\|}{\sigma_{\mathrm{min}}(H)}}\log\frac{1}{\epsilon}} \ .
\end{equation*}
\end{proof}
    
    Combining Theorem \ref{thm:main} and \ref{thm:lb}, it shows that rAPDHG (Algorithm \ref{alg:rapd}) exhibits convergence rate matching the lower complexity on the constructed worst-case instances upto the additional $\|Q\|/\|A\|$ term. Indeed, the term $\|Q\|/\|A\|$ comes from the change of norm, and can be diminished by rescaling the constraints and/or the objective. This implies the optimality of Algorithm \ref{alg:rapd} among span-respecting first-order methods.

\section{Preliminary numerical experiments}\label{sec:numerical}
In this section, we present preliminary numerical experiments, which demonstrate that both restarting and acceleration can improve the performance of the algorithms.

{\bf Dataset.} We utilize 134 instances from the Maros–M\'{e}sz\'{a}ros test set \cite{maros1999repository}, which is a standard benchmark set for convex quadratic programming.

{\bf Progress metric.}  We utilize relative KKT error as the progress metric for different algorithms. For a solution $z=(x,y)$ with $y\geq 0$, the relative primal residual, dual residual and primal-dual gap for \eqref{eq:minimax} are given as follows
{\footnotesize
\begin{equation*}
    \begin{aligned}
        r_{\mathrm{primal}} = \frac{\|[Ax-b]^+\|_{\infty}}{1+\max\{\|Ax\|_{\infty},\|b\|_\infty\}},\  r_{\mathrm{dual}} = \frac{\|Qx+A^Ty+c\|_\infty}{1+\max\{\|Qx\|_{\infty},\|A^Ty\|_\infty,\|c\|_\infty\}},\ r_{\mathrm{gap}}=\frac{|x^TQx+c^Tx+b^Ty|}{1+\max\{|\frac 12 x^TQx+c^Tx|,|\frac 12 x^TQx+b^Ty|\}}
    \end{aligned}
\end{equation*}
}

The relative KKT error is defined as the maximal relative primal residual, dual residual and primal-dual gap 
\begin{equation*}
    \mathrm{relKKT}(z) = \max\{r_{\mathrm{primal}},r_{\mathrm{dual}},r_{\mathrm{gap}}\} \ .
\end{equation*}
The algorithm terminates when relative KKT error is smaller than the termination tolerance $\epsilon$, namely,
\begin{equation}\label{eq:terminate}
    \mathrm{relKKT}(z)\leq \epsilon \ .
\end{equation}
We consider two relative KKT error levels, $\epsilon=10^{-3}$ as low accuracy and $\epsilon=10^{-6}$ as high accuracy. 

{\bf Preprocessing.} FOMs may suffer from slow convergence when solving real-world instances due to their ill-conditioning nature. Before running FOMs, we utilize the diagonal preconditioning heuristic developed in PDLP for linear programming~\cite{applegate2021practical} to rescale the problems for obtaining a better condition number.  More specifically, we rescale the matrix $A$ to $\tilde A=D_1AD_2$ with positive definite diagonal matrices $D_1$ and $D_2$. Both vectors $b$ and $c$ are correspondingly rescaled as $\tilde b=D_1b$ and $\tilde c=D_2c$. The objective matrix $Q$ is also rescaled as $\tilde Q=D_2QD_2$. The matrices $D_1$ and $D_2$ are obtained by running 10 steps of Ruiz scaling \cite{ruiz2001scaling} followed by an $l_2$ scaling and a Pock-Chambolle scaling \cite{pock2011diagonal} on the matrix $\begin{pmatrix}
    Q & A^T \\ A & 0
\end{pmatrix}$. More details about this preconditioning scheme can be found in~\cite{applegate2021practical}.

{\bf Algorithms.} We compare the performance of four algorithms: PDHG~\cite{chambolle2011first}, accelerated PDHG~\cite{chen2014optimal}, restarted PDHG~\cite{applegate2023faster}, and restarted accelerated PDHG (rAPDHG, Algorithm \ref{alg:rapd}). All methods use all-zero vectors as the initial starting points.

We consider three restarting schemes: adaptive restart, tuned fixed frequency restart and no restart. 
\begin{itemize}
    \item No restart: we run the base algorithm (PDHG or accelerated PDHG) without restarting.
    \item Tuned fixed frequency restart: the theoretical guarantee is established for accelerated PDHG with fixed frequency restart. However, we must estimate the quadratic growth $\alpha_\xi$ to choose the restart frequency. Here we run restarted accelerated PDHG several times with frequency selected from $\{10,100,1000\}$ and report the best result.
    \item Adaptive restart: in practice, tuning the restart frequency can be highly costly. Here we propose a heuristic to restart adaptively. In this scheme, we restart when the relative KKT error has a contraction. More specifically, we restart the base algorithm whenever the relative KKT error is halved:
    \begin{equation*}
        \mathrm{relKKT}(\bar z^{n,k})\leq \frac{1}{2} \mathrm{relKKT}(z^{n,0}) \ .
    \end{equation*}
    A similar scheme is also used in the GPU-based LP solver cuPDLP~\cite{lu2023cupdlp,lu2023cupdlpc}.
\end{itemize}

{\bf Shifted geometric mean.} We report the shifted geometric mean of iterations required across instances for each method. More precisely, shifted geometric mean is defined as $\sqrt[N]{\prod_{i=1}^N (n_i+k)}-k$ and we shift by $k=10$. If the instance cannot be solved within iteration limit $200000$, we simply set $n_i=200000$.
\begin{table}[ht!]
\centering
\begin{tabular}{|c|c|c|c|c|}
\hline
          & Algorithm \ref{alg:rapd} & {Accelerated PDHG}    & Restarted PDHG     & {PDHG}                                                \\  \hline
Number of solved instances & \multicolumn{1}{c|}{96}                & 86         & \multicolumn{1}{c|}{89}              & 44         \\ \hline
Number of iterations SGM10     & \multicolumn{1}{c|}{3535.26}     & 7189.47    & \multicolumn{1}{c|}{6333.12}      & 74611.79   \\ \hline
\end{tabular}
\caption{The number of solved instances and shifted geometric mean of the four algorithms for solving 134 Maros–M\'{e}sz\'{a}ros QP instances to low accuracy $\epsilon= 10^{-3}$. {Adaptive restart is used for Algorithm \ref{alg:rapd} and Restarted PDHG.}}
\label{tab:low}
\end{table}

\begin{table}[ht!]
\centering
\begin{tabular}{|c|c|c|c|c|}
\hline
& Algorithm \ref{alg:rapd} & {Accelerated PDHG}    & Restarted PDHG     & {PDHG}  \\ \hline
Number of solved instances & \multicolumn{1}{c|}{69}               & 33         & \multicolumn{1}{c|}{58}                & 1         \\ \hline
Number of iterations SGM10     & \multicolumn{1}{c|}{12803.94}     & 91482.46    & \multicolumn{1}{c|}{22136.86}      & 177231.76   \\ \hline
\end{tabular}
\caption{The number of solved instances and shifted geometric mean of the four algorithms for solving 134 Maros–M\'{e}sz\'{a}ros QP instances to high-accuracy $\epsilon=10^{-6}$. {Adaptive restart is used for Algorithm \ref{alg:rapd} and Restarted PDHG.}}
\label{tab:high}
\end{table}

\begin{figure}[ht!]
    \hspace{-1cm}
	\begin{tabular}{c c c c}
		& \includegraphics[width=0.33\textwidth]{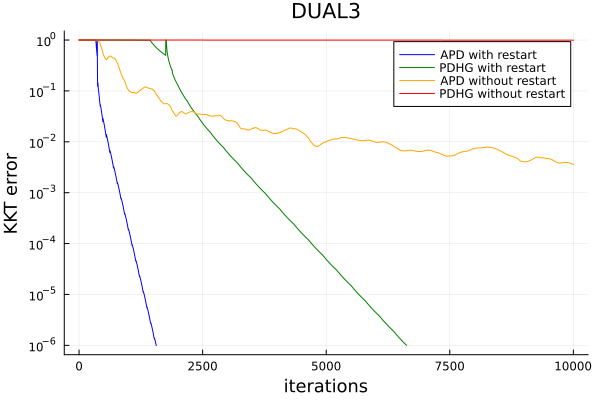}
        & \includegraphics[width=0.33\textwidth]{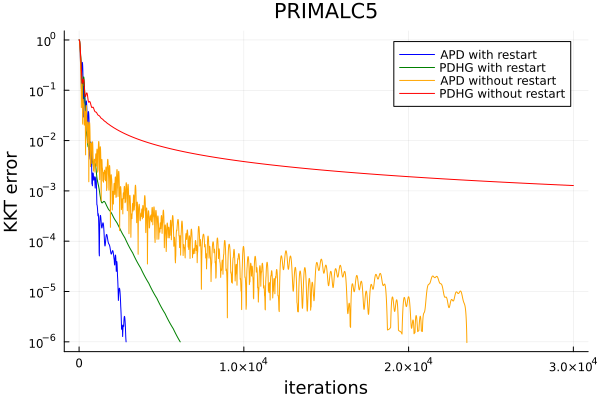}
        & \includegraphics[width=0.33\textwidth]{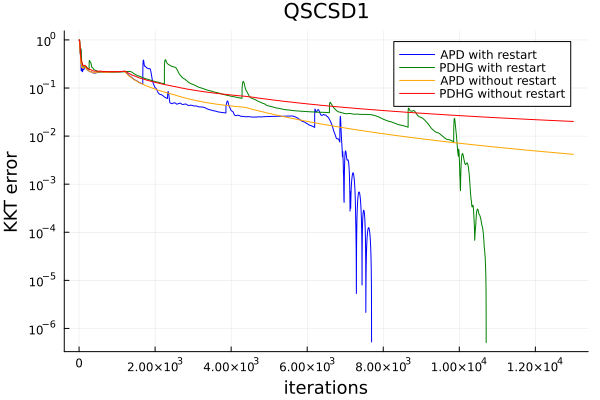}\\
        & \includegraphics[width=0.33\textwidth]{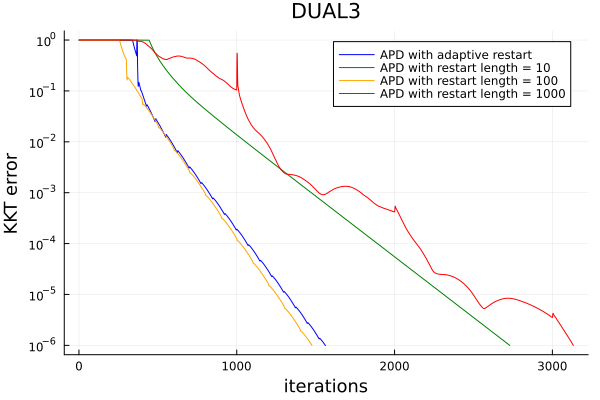}
        & \includegraphics[width=0.33\textwidth]{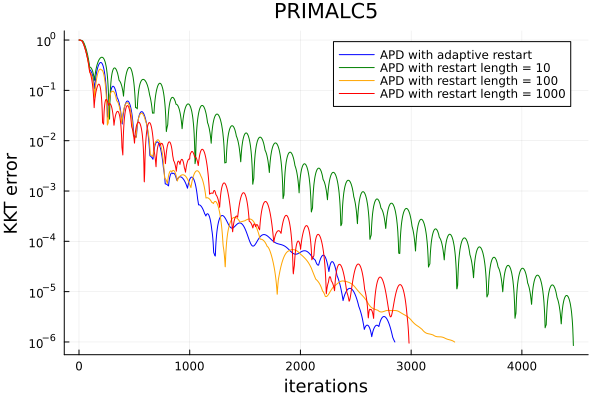}
        & \includegraphics[width=0.33\textwidth]{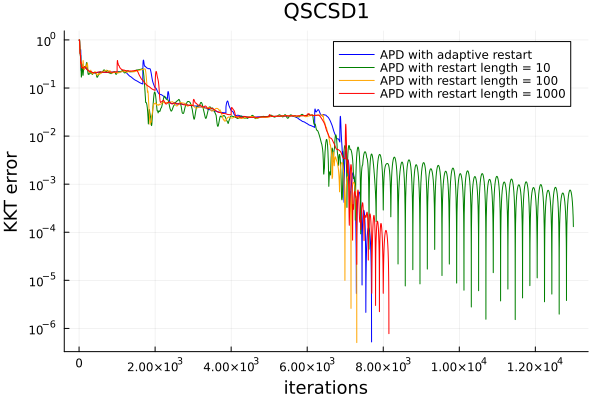}
	\end{tabular}
	\caption{Effects of momentum and restart on convergence behaviors. The first row of figures presents the relative KKT error versus the number of iterations for the four algorithms on the three instances. The second row of figures presents the performance of Algorithm \ref{alg:rapd} with different fixed frequency restart frequencies.}
	\label{fig:real}
\end{figure}

{\bf Results.} Table \ref{tab:low} and \ref{tab:high} summarize the performances of the four algorithms on 134 QP instances from the Maros–M\'{e}sz\'{a}ros test set. The number of solved instances and iterations required are reported. The two tables show that restarting and momentum play a significant role when solving QP. For momentum, in the case of low accuracy, accelerated PDHG without restart solves 86 instances, while PDHG without restart only solves 44. This result is more significant when we look for high-accuracy solutions: only one instance is solved with PDHG, while 33 instances can be solved with accelerated PDHG. Furthermore, it is observed that the adaptive restart scheme boosts the convergence. The iterations of accelerated PDHG with restart are almost half of those of restarted PDHG, and meanwhile, restarted accelerated PDHG can solve more instances than restarted PDHG in both low and high-accuracy cases. To sum up, the experiments demonstrate the effectiveness of our algorithmic enhancements on top of PDHG, i.e., momentum and restart, for solving QP problems, and they showcase the superior numerical performances of Algorithm \ref{alg:rapd}.

Figure \ref{fig:real} presents the effects of momentum and restart on the convergence behavior of algorithms. We choose three representative instances, \texttt{DUAL3}, \texttt{PRIMALC5}, and \texttt{QSCSD1} from Maros–M\'{e}sz\'{a}ros test set to illustrate their behaviors, and we observe similar performance in most of the other tested instances. In particular, the significant effect of momentum can be clearly observed, as shown in the first row of Figure \ref{fig:real}. Without restart, vanilla accelerated PDHG and vanilla PDHG are hard to achieve $10^{-3}$, except for accelerated PDHG on \texttt{PRIMALC5}. Besides, in the presence of restarting, momentum helps to reduce the needed iterations significantly. Furthermore, in the first row of Figure \ref{fig:real}, we see both PDHG and accelerated PDHG suffer slow sublinear convergence. In contrast, restarts trigger linear convergence of both accelerated PDHG and PDHG. In the second row of Figure \ref{fig:real}, we compare the adaptive restart scheme with three fixed restart frequencies. Notably, the number of iterations is fairly similar for the adaptive scheme and the best-fixed frequency. The benefit of adaptive restart comes from not searching over the restart lengths and thus requires much lower computational cost.

\section{PDQP.jl: A Julia implementation of rAPDHG for QP}\label{sec:PDQP}
In this section, we present PDQP.jl, a QP solver based on rAPDHG supporting both GPU and CPU. Unlike the preliminary experiments stated in the previous section, whose goal is to demonstrate the effectiveness of acceleration and restarting, the implementation and experiments stated in this section are much more sophisticated. The implementation of the solver is available at \href{https://github.com/jinwen-yang/PDQP.jl}{https://github.com/jinwen-yang/PDQP.jl}. We compare the numerical performance of PDQP and two popular QP solvers, SCS and OSQP. Notice that PDQP.jl is a preliminary prototype implementation in Julia, and the goal here is to demonstrate the potential of the proposed algorithm. We anticipate a more sophisticated C/C++ implementation that will have superior performance to PDQP.jl.

In PDQP, we utilize several algorithmic enhancements upon Algorithm \ref{alg:rapd} to improve its practical performance. More specifically,
\begin{itemize}
    \item {\bf Preprocessing.} The same strategy of preprocessing as discussed in Section \ref{sec:numerical} is employed in PDQP.
    \item {\bf Adaptive restart.} PDQP adopts similar adaptive restarting strategy of PDLP~\cite{applegate2021practical,lu2023cupdlp}. The restart scheme is described as follows: Define the restart candidate as 
\begin{equation*}
    z_c^{n,t+1}:=\mathrm{GetRestartCandidate}(z^{n,t+1},\bar z^{n,t+1})=\begin{cases}
        z^{n,t+1},& \mathrm{relKKT}(z^{n,t+1})<\mathrm{relKKT}(\bar z^{n,t+1}) \\ \bar z^{n,t+1},& \mathrm{otherwise} \ .
    \end{cases}
\end{equation*} 
The algorithm restarts if one of three conditions holds:
\begin{enumerate}
    \item[(i)] (Sufficient decay in relative KKT error)
    \begin{equation*}
        \mathrm{relKKT}(z_c^{n,t+1})\leq \beta_{\mathrm{sufficient}}\mathrm{relKKT}(z^{n,0}) \ ,
    \end{equation*}
    \item[(ii)] (Necessary decay + no local progress in relative KKT error)
    \begin{equation*}
        \mathrm{relKKT}(z_c^{n,t+1})\leq \beta_{\mathrm{necessary}}\mathrm{relKKT}(z^{n,0})\quad \text{and} \quad \mathrm{relKKT}(z_c^{n,t+1})>\mathrm{relKKT}(z_c^{n,t}) \ ,
    \end{equation*}
    \item[(iii)] (Long inner loop)
    \begin{equation*}
        t\geq \beta_{\mathrm{artificial}}k \ ,
    \end{equation*}
\end{enumerate}
where {$k$ is the iteration number of the current epoch} and parameters $\beta_{\mathrm{sufficient}}=0.2$, $\beta_{\mathrm{necessary}}=0.8$ and $\beta_{\mathrm{artificial}}=0.36$.

    \item {\bf Adaptive step-size and primal weight.} The primal and the dual step-sizes of PDQP are reparameterized as
    \begin{equation*}
        \tau = \eta/\omega,\; \sigma=\eta\omega\quad \text{with}\;\  \eta,\omega>0\ ,
    \end{equation*}
    where $\eta$ (called step-size) controls the scale of the step-sizes, and $\omega$ (called primal weight) balances the primal and the dual progress. The step-size $\eta^{n,k}$ at inner iteration $k$ is selected as 
    \begin{equation*}
        \eta^{n,k} = \begin{cases}
            \frac{1.98}{\frac{\|Q\|}{\omega}+\sqrt{4\|A\|^2+\frac{\|Q\|^2}{\omega^2}}}, & k=0\\
            \min\left\{  \pran{1+\frac{1}{k}}\eta^{n,k-1}, \frac{0.99(k+2)}{\frac{\|Q\|}{\omega}+\sqrt{\|A\|^2(k+2)^2+\frac{\|Q\|^2}{\omega^2}}}  \right\}, & k \geq 1
        \end{cases}
    \end{equation*}
    while the update of primal weight is specific during restart occurrences, thus infrequently. More precisely, the initialization of $\omega$ involves the expression:
    \begin{equation*}
    \mathrm{InitializePrimalWeight}(c,b):=\begin{cases}
        \frac{\|c\|_2}{\|b\|_2},\; & \text{if } \|c\|_2,\|b\|_2>\epsilon_{\mathrm{zero}}\\
        1, \; & \mathrm{otherwise}
    \end{cases}
    \end{equation*}
    where $\epsilon_{\mathrm{zero}}$ denotes a small nonzero tolerance. Let $\Delta_x^n=\|x^{n,0}-x^{n-1,0}\|_2$ and $\Delta_y^n=\|y^{n,0}-y^{n-1,0}\|_2$. PDQP initiates the primal weight update at the beginning of each new epoch with $\theta=0.2$.
    {\small
    \begin{equation*}
    \mathrm{PrimalWeightUpdate}(z^{n,0},z^{n-1,0},\omega^{n-1}):=\begin{cases}
        \exp\pran{\theta \log\pran{\frac{\Delta_y^n}{\Delta_x^n}}+(1-\theta)\log\omega^{n-1}},\; & \Delta_x^n,\Delta_y^n>\epsilon_{\mathrm{zero}}\\
        \omega^{n-1}, \; & \mathrm{otherwise}
    \end{cases}
    \end{equation*}
    }
    The adaptive step-size rule and the primal weight update are extensions of those in PDLP \cite{applegate2021practical} to the QP case.
\end{itemize}

{
While PDQP's enhancements are theoretically motivated, the focus of this section and the PDQP solver is on practical performance and some of these modifications may not preserve the original theoretical guarantees, as detailed below (similar gap between theory and practice also appears in other solvers, such as PDLP~\cite{applegate2021practical}):
    \begin{itemize}
        \item Diagonal preconditioning preserves theoretical guarantees since it can be interpreted as applying rAPDHG to a QP instance with modified data.
        \item Adaptive restart does not readily preserve convergence guarantee. Still, we conjecture that a proof of convergence could be established based on the theoretical guarantees derived for restarting with a fixed frequency.
        \item Currently, we do not have a proof of convergence for the adaptive step size rule.
    \end{itemize}
}

Next, we present the numerical performance of PDQP and compare it with that of SCS and OSQP.

\subsection{Experimental setup}\label{sec:experiment-setup}

{\bf Dataset.} We utilize 134 instances from the Maros–M\'{e}sz\'{a}ros test set \cite{maros1999repository}, which is a standard benchmark set for convex quadratic programming. We also collect 33 convex QP instances from QPLIB~\cite{furini2019qplib}. To compare solver performances on larger instances, the selection of instances is refined to instances with $\text{number of variables}+ \text{number of constraints} > 10000$, and 50 are selected in total to curate a medium-sized dataset. The size of the instances is summarized in Table \ref{tab:size}. Indeed, both benchmark sets are rather small, and QPLIB is relatively larger than Maros–M\'{e}sz\'{a}ros test set.

\begin{table}[ht!]
\centering
{\small
\begin{tabular}{c|ccccc}
\hline
                    
\textbf{Instance Size}  & \textless 100  & 100-1,000 & 1,000-10,000 & 10,000-100,000 & 100,000-1,000,000 \\\hline
\textbf{Maros–M\'{e}sz\'{a}ros} & 19                     & 41                         & 42       & 30  & 2  \\
\textbf{QPLIB} & 0                     & 5                         & 10      & 14  & 4 
\\ \hline
\end{tabular}
}
\caption{Scales of instances in Maros–M\'{e}sz\'{a}ros and QPLIB benchmark sets. The instance size here represents the sum of $\text{number of variables}$ and $\text{number of constraints}$. }
\label{tab:size}
\end{table}

Maros–M\'{e}sz\'{a}ros and QPLIB benchmark sets are outdated and do not contain large instances. {To complement the results on (relatively small-sized) standard QP benchmark datasets, we consider 63 large-scale synthetic instances from 7 problem classes as in \cite[Section 8]{stellato2020osqp} in Section \ref{sec:large-qp}. The problem classes include random QP, random QP with equality constraints, control problem, portfolio optimization, Huber regression, Lasso and SVM.  For each class of problems, we generated 9 different random instances, and based on their number of nonzeros in objective and constraint matrices, the instances within each class are categorized into small (nonzeros $\approx$ 300K), medium (nonzeros $\approx$ 3M) and large (nonzeros $\approx$ 30M). }

{\bf Computing environment.} We use NVIDIA H100-PCIe-80GB GPU, with CUDA 12.3, for running GPU version of solvers, and we use Intel Xeon Gold 6248R CPU 3.00GHz with 160GB RAM for running CPU-based solvers.  The experiments are performed in Julia 1.9.2.

{\bf Solvers.} PDQP.jl is implemented in an open-source Julia~\cite{bezanson2017julia} module. The GPU implementation of PDQP.jl utilizes \href{https://github.com/JuliaGPU/CUDA.jl}{CUDA.jl}~\cite{besard2018effective} as the interface for working with NVIDIA CUDA GPUs using Julia. We compare PDQP.jl with two open-sourced QP solves: SCS~\cite{o2016conic} and OSQP~\cite{stellato2020osqp},{as well as barrier method implemented in Gurobi 11.0}. Three different implementations of SCS are considered: direct method on CPU, indirect method on CPU and GPU version.

{\bf Termination.} PDQP utilizes relative KKT error~\eqref{eq:KKT} for termination, which checks primal feasibility, dual feasibility and primal-dual gap. This is almost an identical termination criteria as SCS~\cite{o2016conic,o2021operator}. On the other hand, OSQP only checks primal and dual feasibility and does not allow the user to specify a bound on the primal-dual gap. As a result, OSQP may not always provide reliable solutions (see~\cite{o2021operator} for examples and discussions on the lack of gap check of OSQP). {To confirm that the gap falls within the required tolerance for a fair comparison, we adopt the same procedure as in \cite{o2021operator}. Initially, we solved the problem using pre-set tolerances. If the solution meets the gap tolerance, we accept it. If not, we will reduce the tolerance by half and solve the problem again. This process is repeated until a solution that satisfies the gap constraint is found, and only the final solution is included in the statistics.} Notice that SCS and OSQP both utilize ADMM-based algorithms, and it was shown in \cite{o2021operator} that SCS has superior performance than OSQP for QP; thus, we focus on comparing PDQP.jl with SCS. {We set $10^{-3}$ and $10^{-6}$ tolerances for parameters \texttt{FeasibilityTol}, \texttt{OptimalityTol} and \texttt{BarConvTol} for barrier methods of Gurobi.}

{\bf Time limit.} We impose a time limit of 3600 seconds on all the instances. 

{\bf Shifted geometric mean.} We report the shifted geometric mean of solve time required across instances for each solver. More precisely, shifted geometric mean is defined as $\sqrt[N]{\prod_{i=1}^N (t_i+k)}-k$ and we shift by $k=10$. If the instance cannot be solved within time limit $3600$ seconds, we simply set $t_i=3600$.

\subsection{Results on two benchmark sets}

\begin{table}[ht!]
\centering
\begin{tabular}{ccccc}
\hline
\multicolumn{1}{l}{}                                & \multicolumn{2}{c}{\textbf{Tol 1E-03}}      & \multicolumn{2}{c}{\textbf{Tol 1E-06}}       \\ 
                                                    & \textbf{Count} & \textbf{Time} & \textbf{Count} & \textbf{Time} \\ \hline
\multicolumn{1}{c}{\textbf{PDQP (GPU)}} & 131                   & 8.06              & 122                    & 31.26              \\
\multicolumn{1}{c}{\textbf{PDQP (CPU)}}     & 131                    & 8.40               & 118                    & 26.81             \\
\multicolumn{1}{c}{\textbf{SCS (GPU)}}        & 130                    & 17.33             & 108                   & 108.47              \\
\multicolumn{1}{c}{\textbf{SCS (CPU-indirect)}}   & 130                    & 8.90              & 113                    & 64.63             \\
\multicolumn{1}{c}{\textbf{SCS (CPU-direct)}}        & 131                    & 2.75             & 126                   & 16.86              \\
\multicolumn{1}{c}{\textbf{OSQP (CPU)}}        & 125                  & 10.27             & 110                 & 44.03         \\
\hdashline
\multicolumn{1}{c}{\textbf{{ Gurobi (barrier)}}}       & {134}                  & {0.27}            & {132}                 & {1.25}  \\
\hline
\end{tabular}
\caption{Solve time in seconds and SGM10 of different solvers on 134 instances of Maros–M\'{e}sz\'{a}ros benchmark set with tolerance $10^{-3}$ and $10^{-6}$.}
\label{tab:mm}
\end{table}

Table \ref{tab:mm} presents a comparison between PDQP and other two open-source solvers SCS and OSQP on Maros–M\'{e}sz\'{a}ros benchmark set, yielding several noteworthy observations:
\begin{itemize}
    \item PDQP (CPU and GPU) has stronger performance than SCS-indirect methods (CPU and GPU) regarding solved count and solve time. In the case of high accuracy ($\epsilon=10^{-6}$), PDQP establishes an advantage over indirect version of SCS, achieving a 2.4x speedup on CPU implementation and a 3.5x speedup on GPU version.
    \item SCS (CPU-direct) has superior performance than PDQP particularly for high accuracy. This is expected as the instances in Maros–M\'{e}sz\'{a}ros benchmark set are small and factorization used in the direct method is generally cheap.
    \item {The GPU implementation of PDQP has a similar performance as the CPU implementation. This is due to the small size of benchmark instances.}
     \item OSQP has inferior performance to other solvers on Maros–M\'{e}sz\'{a}ros benchmark set when the same termination criteria are used, which is consistent with the results in~\cite{o2021operator}. 
    \item {Barrier method in Gurobi is the most efficient method, given the small sizes of the instances in these benchmark sets.}
\end{itemize}

\begin{table}[ht!]
\centering
\begin{tabular}{ccccc}
\hline
\multicolumn{1}{l}{}                                & \multicolumn{2}{c}{\textbf{Tol 1E-03}}      & \multicolumn{2}{c}{\textbf{Tol 1E-06}}       \\ 
                                                    & \textbf{Count} & \textbf{Time} & \textbf{Count} & \textbf{Time} \\ \hline
\multicolumn{1}{c}{\textbf{PDQP (GPU)}} & 28                   & 23.84              & 27                   & 49.86             \\
\multicolumn{1}{c}{\textbf{PDQP (CPU)}}     & 28                    & 35.83               & 26                    & 74.64             \\
\multicolumn{1}{c}{\textbf{SCS (GPU)}}        & 27                    & 65.41           & 23                   & 245.59             \\
\multicolumn{1}{c}{\textbf{SCS (CPU-indirect)}}   & 27                    & 56.91              & 22                    & 275.27          \\
\multicolumn{1}{c}{\textbf{SCS (CPU-direct)}}        & 28                    & 26.86            & 26                  & 67.81              \\
\multicolumn{1}{c}{\textbf{OSQP (CPU)}}        & 26                    & 66.21             & 22                & 138.97              \\
\hdashline
\multicolumn{1}{c}{\textbf{{ Gurobi (barrier)}}}       & {32}                  & {4.59}            & {32}                 & {4.77}  \\
\hline
\end{tabular}
\caption{Solve time in seconds and SGM10 of different solvers on 33 instances of QPLIB with tolerance $10^{-3}$ and $10^{-6}$.}
\label{tab:qplib}
\end{table}

Table \ref{tab:qplib} presents the result for QPLIB (whose instances are relatively larger than Maros–M\'{e}sz\'{a}ros benchmark set), here are a few observations:
\begin{itemize}
    \item PDQP (CPU and GPU) has stronger performance than SCS indirect (CPU and GPU) on QPLIB. In the case of high accuracy ($\epsilon=10^{-6}$), PDQP establishes an advantage over indirect version of SCS, achieving a 3.7x speedup on CPU implementation with 4 more instances solved and a 5x speedup on GPU version with 4 more solved. In addition, PDQP (GPU) has comparable performance to and is faster than SCS (direct) under low accuracy and high accuracy.
    \item The GPU implementation of PDQP has better performance than CPU implementation due to the (relatively) larger size of instances.
\end{itemize}

\begin{table}[ht!]
\centering
\begin{tabular}{ccccc}
\hline
\multicolumn{1}{l}{}                                & \multicolumn{2}{c}{\textbf{Tol 1E-03}}      & \multicolumn{2}{c}{\textbf{Tol 1E-06}}       \\ 
                                                    & \textbf{Count} & \textbf{Time} & \textbf{Count} & \textbf{Time} \\ \hline
\multicolumn{1}{c}{\textbf{PDQP (GPU)}} & 47                   & 18.22             & 38                   & 159.38             \\
\multicolumn{1}{c}{\textbf{PDQP (CPU)}}     & 46                    & 39.46               & 33                    & 303.83             \\
\multicolumn{1}{c}{\textbf{SCS (GPU)}}        & 42                    & 67.12             & 29                   & 603.93              \\
\multicolumn{1}{c}{\textbf{SCS (CPU-indirect)}}   & 42                    & 64.90              & 28                    & 771.07         \\
\multicolumn{1}{c}{\textbf{SCS (CPU-direct)}}        & 44                    & 23.89             & 38                  & 153.10              \\
\multicolumn{1}{c}{\textbf{OSQP (CPU)}}        & 41                  & 68.14           & 36                  & 250.47              \\
\hdashline
\multicolumn{1}{c}{\textbf{{ Gurobi (barrier)}}}       & {49}                  & {3.11}            & {47}                 & {6.82}  \\
\hline
\end{tabular}
\caption{Solve time in seconds and SGM10 of different solvers on 50 medium-sized instances with tolerance $10^{-3}$ and $10^{-6}$.}
\label{tab:medium}
\end{table}

We further refine the selection of instances to curate the medium-sized instances. The results are summarized in Table \ref{tab:medium}. We can clearly see a superior performance of GPU version of PDQP compared to its CPU version, particularly for high accuracy. This demonstrates the power of GPU implementation in solving larger-sized instances. Furthermore, we can observe that the PDQP (GPU) performs better than SCS on relatively larger instances, particularly the indirect methods. Compared to the direct method of SCS, it can also be observed that PDQP has advantages on larger instances, which is consistent with the intuition that matrix-free first-order methods can scale better than algorithms that need to perform matrix factorization.

{

\subsection{Large-scale synthetic QP}\label{sec:large-qp}

\begin{table}[ht!]
\centering
\begin{tabular}{cccccccccc}
\hline
\multirow{2}{*}{}                                   & \multicolumn{2}{c}{\begin{tabular}[c]{@{}c@{}}\textbf{Small (21)}\end{tabular}} & \multicolumn{2}{c}{\begin{tabular}[c]{@{}c@{}}\textbf{Medium (21)}\end{tabular}}    & \multicolumn{2}{c}{\begin{tabular}[c]{@{}c@{}}\textbf{\textbf{Large (21)}}\end{tabular}}   & \multicolumn{2}{c}{\textbf{Total (63)}}       \\
                                                                                                   & \textbf{Count} & \textbf{Time} & \textbf{Count} & \textbf{Time} & \textbf{Count} & \textbf{Time}  & \textbf{Count} & \textbf{Time} \\ \hline
\multicolumn{1}{c}{\textbf{PDQP (GPU)}}     & 21                   & 1.20               & 21                    & 1.94               & 21                    & 6.13 &63 &2.92             \\
\multicolumn{1}{c}{\textbf{\begin{tabular}[c]{@{}c@{}}PDQP (CPU)\end{tabular}}}               & 21                   & 3.01               & 21                    & 27.53              & 18                    & 359.31  & 60& 46.49           \\
\multicolumn{1}{c}{\textbf{\begin{tabular}[c]{@{}c@{}}SCS (GPU)\end{tabular}}} & 21                   & 2.47               & 21                    & 10.02             & 21                    & 68.54   &    63 & 16.97       \\
\multicolumn{1}{c}{\textbf{\begin{tabular}[c]{@{}c@{}}SCS (CPU-indirect)\end{tabular}}}      & 21                   & 9.39               & 21                    & 81.86               & 15                    & 700.88    & 57 & 98.18  
\\
\multicolumn{1}{c}{\textbf{\begin{tabular}[c]{@{}c@{}}SCS (CPU-direct)\end{tabular}}}      & 21                   & 1.06               & 21                    & 10.19               & 21                    & 133.52     & 63 & 21.76       \\ 
\multicolumn{1}{c}{\textbf{\begin{tabular}[c]{@{}c@{}}OSQP (CPU)\end{tabular}}}      & 21                   & 1.11               & 21                    & 11.80              & 21                    & 170.71     & 63 & 25.24       \\
\hdashline
\multicolumn{1}{c}{\textbf{\begin{tabular}[c]{@{}c@{}}{Gurobi (barrier)}\end{tabular}}}      & {21}                  & {1.02}            &{ 21}                &{9.73}          &{18}                 &{124.01}    & {60} & {20.77}      \\
\hline
\end{tabular}
\caption{Solve time in seconds and SGM10 of different solvers on instances of with tolerance $10^{-3}$.}
\label{tab:synthetic-1e-3}
\end{table}

\begin{table}[ht!]
\centering
\begin{tabular}{cccccccccc}
\hline
\multirow{2}{*}{}                                   & \multicolumn{2}{c}{\begin{tabular}[c]{@{}c@{}}\textbf{Small (21)}\end{tabular}} & \multicolumn{2}{c}{\begin{tabular}[c]{@{}c@{}}\textbf{Medium (21)}\end{tabular}}    & \multicolumn{2}{c}{\begin{tabular}[c]{@{}c@{}}\textbf{\textbf{Large (21)}}\end{tabular}}   & \multicolumn{2}{c}{\textbf{Total (63)}}       \\
                                                                                                   & \textbf{Count} & \textbf{Time} & \textbf{Count} & \textbf{Time} & \textbf{Count} & \textbf{Time}  & \textbf{Count} & \textbf{Time} \\ \hline
\multicolumn{1}{c}{\textbf{PDQP (GPU)}}     & 21                   & 2.08             & 21                    & 3.40               & 21                    & 12.17 &63 &5.31             \\
\multicolumn{1}{c}{\textbf{\begin{tabular}[c]{@{}c@{}}PDQP (CPU)\end{tabular}}}               & 21                   & 5.55             & 21                    & 54.18            & 17                   & 647.42  & 59& 76.89          \\
\multicolumn{1}{c}{\textbf{\begin{tabular}[c]{@{}c@{}}SCS (GPU)\end{tabular}}} & 21                   & 5.69               & 21                    & 24.38             & 18                    & 196.75   &    60 & 38.14       \\
\multicolumn{1}{c}{\textbf{\begin{tabular}[c]{@{}c@{}}SCS (CPU-indirect)\end{tabular}}}      & 21                   & 23.90              & 20                    & 267.27             & 12                    & 1593.95   & 53 & 237.04 
\\
\multicolumn{1}{c}{\textbf{\begin{tabular}[c]{@{}c@{}}SCS (CPU-direct)\end{tabular}}}      & 21                   & 3.09              & 21                    & 30.25              & 20                    & 395.85     & 62 & 49.79       \\ 
\multicolumn{1}{c}{\textbf{\begin{tabular}[c]{@{}c@{}}OSQP (CPU)\end{tabular}}}      & 21                   & 3.31              & 21                    & 30.70              & 21                    & 375.24     & 63 & 49.32       \\
\hdashline
\multicolumn{1}{c}{\textbf{\begin{tabular}[c]{@{}c@{}}{Gurobi (barrier)}\end{tabular}}}      & {21}                   & {1.17}               & {21}                    & {10.48}              & {18}                    & {131.85}     & {60} & {21.90}       \\
\hline
\end{tabular}
\caption{Solve time in seconds and SGM10 of different solvers on instances of with tolerance $10^{-6}$.}
\label{tab:synthetic-1e-6}
\end{table}

Table \ref{tab:synthetic-1e-3} and Table \ref{tab:synthetic-1e-6} present the comparison of PDQP with SCS and OSQP. In particular, Table \ref{tab:synthetic-1e-3} presents low accuracy results and Table \ref{tab:synthetic-1e-6} shows the high accuracy results. In case of low accuracy ($\epsilon=10^{-3}$), PDQP on GPU achieves at least 5x speedup on medium-scale instances and enjoys more significant speedup on large-scale instances. Overall, PDQP on GPU has a 5.8x speedup over the second best of all solvers, i.e., SCS on GPU. The speedup of PDQP (GPU) is even more remarkable when we seek solution with high accuracy. Moreover, significant speedup of PDQP (GPU) can also be achieved over Gurobi on large-scale instances.

\begin{figure}[ht!]
    \hspace{-1cm}
	\begin{tabular}{c c c c}
		& \includegraphics[width=0.5\textwidth]{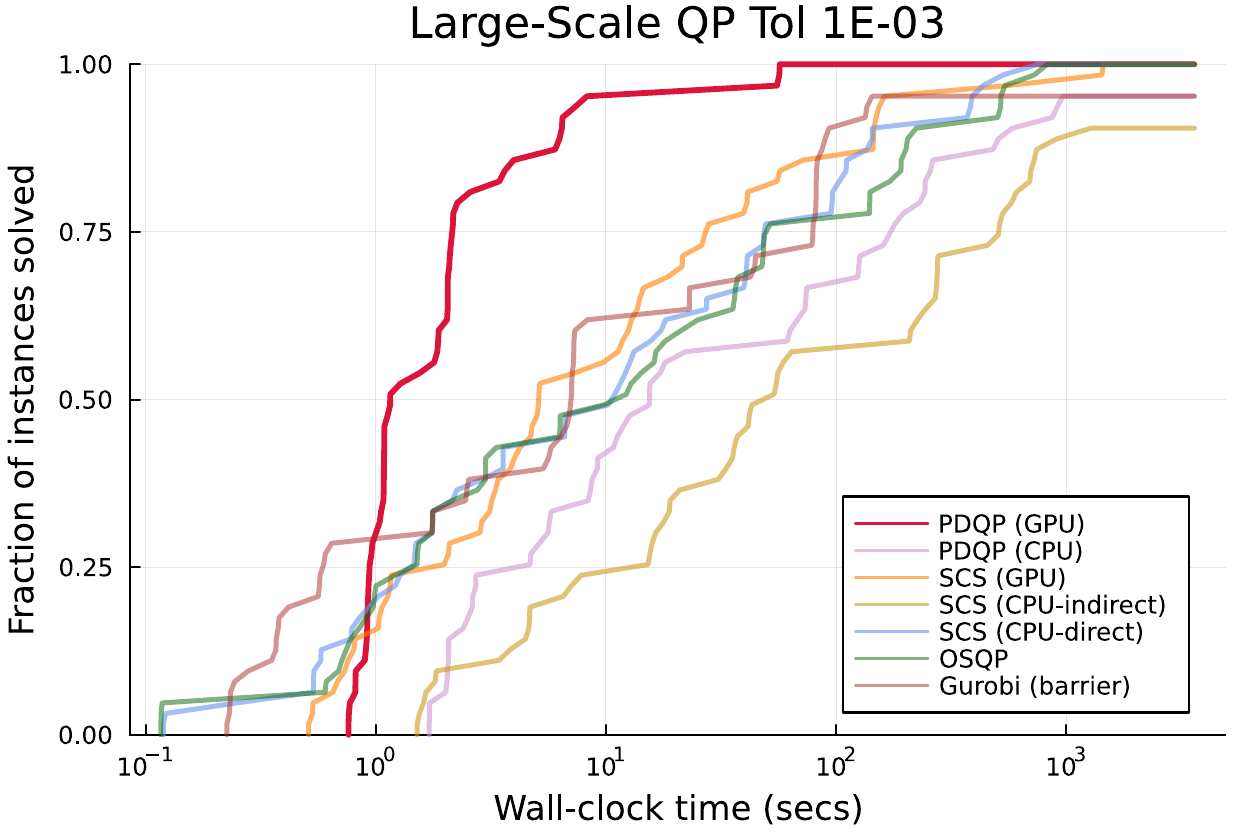}
        & \includegraphics[width=0.5\textwidth]{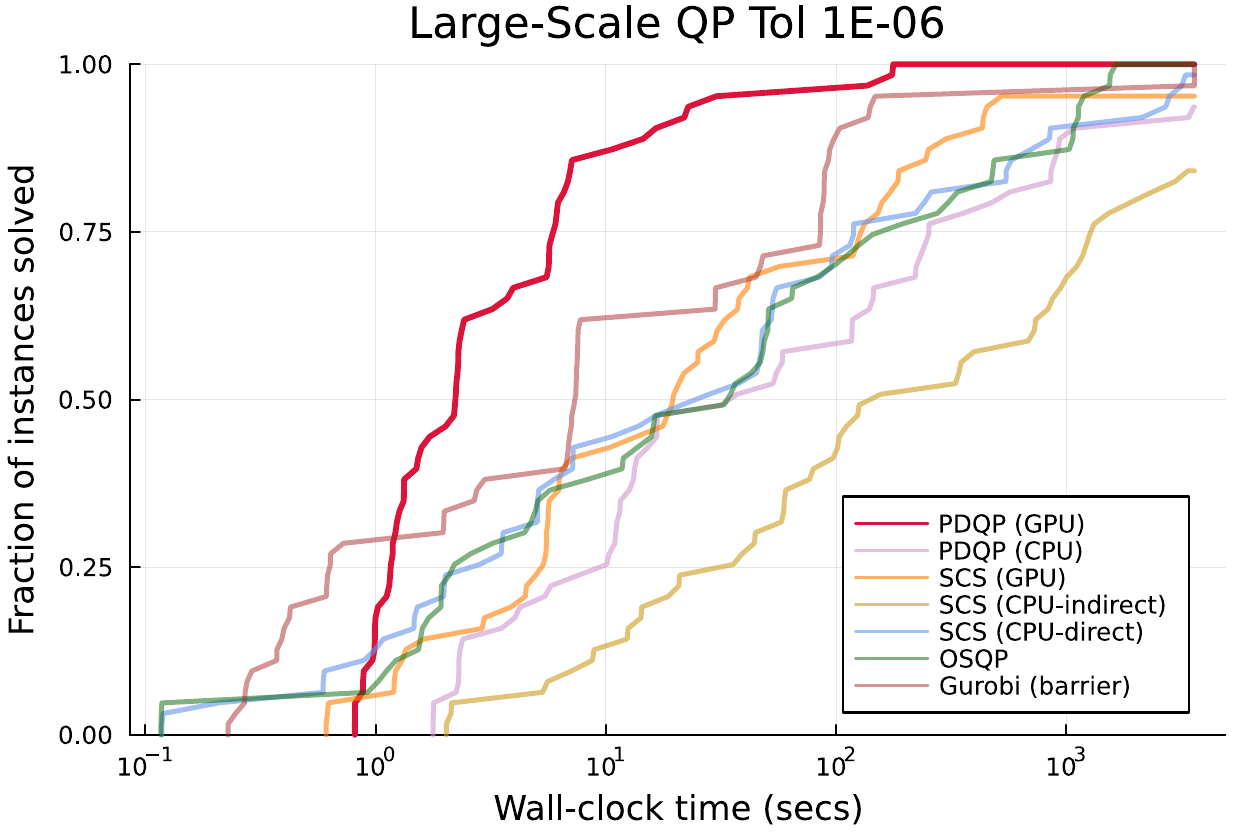}
	\end{tabular}
	\caption{Number of instances solved under low accuracy (left) and high accuracy (right).}
	\label{fig:performance}
\end{figure}

Figure \ref{fig:performance} shows the number of solved instances of PDQP, SCS and OSQP in a given time. The x-axes display the wall-clock time in seconds while the y-axes display the fraction of solved instances. As shown in both panels, PDQP on GPU solves more instances than all other solvers within one second and solves all 63 instances within around a hundred seconds. In contrast, there are several instances where both SCS and OSQP take over a thousand seconds to solve, and instances that takes several hundreds of seconds for Gurobi to solve.

In summary, the GPU-implemented PDQP consistently outperforms SCS, OSQP and Gurobi on these QP instances and PDQP (GPU) demonstrates even more pronounced advantages on instances of medium to large scale.
}

\section{Conclusion}
In this paper, we present rAPDHG, a practical and optimal first-order method for convex QP. We show that the algorithm achieves a linear convergence rate for convex QP, and present lower bound instances demonstrating no first-order methods can achieve a faster rate (up to a constant). We further develop PDQP.jl, an implementation of rAPDHG for QP in Julia with both CPU and GPU versions. Numerical experiments demonstrate its superior performance over SCS and OSQP on larger instances.

\section*{Acknowledgment}
This project is supported by AFOSR grant FA9550-24-1-0051. Haihao Lu is supported by AFOSR grant FA9550-24-1-0051 and ONR grant N000142412735. Jinwen Yang is supported by AFOSR grant FA9550-24-1-0051.

\bibliographystyle{amsplain}
\bibliography{ref-papers}

\end{document}